\documentclass{amsart}

\usepackage{amsmath, amssymb, amsthm, color}
\usepackage[all]{xy}
\usepackage {enumitem}

\newtheorem{thm}{Theorem}[section]
\newtheorem*{introthm}{Theorem}
\newtheorem{lemma}[thm]{Lemma}
\newtheorem{prop}[thm]{Proposition}
\newtheorem{cor}[thm]{Corollary}

\theoremstyle{definition}
\newtheorem{defn}[thm]{Definition}

\newtheorem{remark}[thm]{Remark}

\DeclareMathOperator{\Char}{Char}

\DeclareMathOperator{\Ass}{Ass}
\DeclareMathOperator{\Ht}{Height}
\DeclareMathOperator{\Spec}{Spec}

\newcommand{\R}{\widehat R}
\newcommand{\rt}{\widetilde r}
\newcommand{\of}{\overline f}
\newcommand{\og}{\overline g}
\newcommand{\oog}{\overline{\overline g}}
\newcommand{\oh}{\overline h}
\newcommand{\ooh}{\overline{\overline h}}

\newcommand{\m}{\mathfrak{m}}

\renewcommand{\phi}{\varphi}

\DeclareMathOperator{\Ker}{Ker}
\DeclareMathOperator{\Image}{Image}

\renewcommand{\to}{\longrightarrow}

\title{Are complete intersections complete intersections?}
\author{Raymond C. Heitmann and David A. Jorgensen}
\address{Raymond C. Heitmann \\ Department of Mathematics \\ University of Texas at Austin \\
Austin, TX \\ 78712}
\address{David A. Jorgensen \\ Department of Mathematics \\ University of Texas at Arlington \\
Arlington, TX \\ 76019}
\date{\today}
\keywords{Complete intersection, completion}
\subjclass[2000]{13J10, 13C40, 14M10}

\begin{document}

\thanks{The first author was partially supported by NSF grant DMS-0856124. The second author was partially supported by NSA grant H98230-10-0197}

\begin{abstract}
A commutative local ring is generally defined to be a complete intersection if its completion is
isomorphic to the quotient of a regular local ring by an ideal generated by a regular sequence.
It has not previously been determined whether or not such a ring is necessarily itself the quotient of a regular ring by an ideal generated by a regular sequence.
In this article, it is shown that if a complete intersection is a one dimensional integral domain, then it is such a quotient.
However, an example is produced of a three dimensional complete intersection domain which is not a
homomorphic image of a regular local ring, and so the property does not hold in general.
\end{abstract}

\maketitle

\section*{Introduction}

Let $R$ be a commutative local Noetherian ring, and
let $\widehat R$ denote the completion of $R$ with respect to the
topology defined by the maximal ideal of $R$.
Let us say that the \emph{absolute} definition of $R$ being a complete intersection is that $R$ is
isomorphic to the quotient of a regular local ring by an ideal generated
by a regular sequence,
and let us say that the \emph{formal} definition of $R$ being a complete intersection is that $\R$ is
a complete intersection in the absolute sense.
The formal definition of complete intersection was given in 1967 by Grothendieck in
\cite[19.3.1]{ega}.
Since completion with respect to the maximal ideal of a Noetherian local
ring defines a faithfully flat functor, this definition works well for
most applications and has been predominantly adopted in
recent years as \emph{the} definition of complete intersection in commutative ring theory.
Nevertheless, the absolute definition also occurs frequently in the literature.
In spite of one's preference of definition, it has remained a bane that it is unknown whether the absolute and formal notions of complete intersection are in fact equivalent.
In this paper we give answers which shed light on this question.

In Section 1 below we show that the two notions are the same when
$R$ is a one dimensional integral domain (of arbitrary codimension).  On the
other hand, we show in Section 2 by way of example that there are
commutative local Noetherian rings $R$ whose completion is isomorphic to the quotient of a
regular local ring modulo an ideal generated by a regular sequence, but the ring
itself is not the homomorphic image of a regular local ring. We note that it is well-known
(see, for example, \cite[19.3.2]{ega})
that if $R$ is a complete intersection in the formal sense, then it being the homomorphic image of a regular local ring is equivalent to it being a complete intersection in the absolute sense.  In the
example, $\R=\mathbf{R}[[x,y,z,w]]/(x^2+y^2)$, and so $R$ is an integral domain
of dimension three.  The main question thus remains unanswered when $R$ is of dimension 2,
or dimension 1 and not an integral domain.

The proofs in both Sections 1 and 2 rely on the following basic theorem.

\begin{introthm} Consider a diagram of commutative local ring homomorphisms
\[
\xymatrixrowsep{2pc}
\xymatrixcolsep{4pc}
\xymatrix{
{} & T \ar@{->>}[d]^{\pi}\\
R \ar[r]^\subseteq & \R \\
}
\]
where $(T,\m)$ is a complete regular local ring with dimension equaling the embedding dimension of $R$, and $\pi$ is surjective with kernel generated by a regular sequence contained in $\m^2$.
Then $R$ is isomorphic to the quotient of a regular local ring by an ideal generated by a regular
sequence if one can complete this diagram to a commutative diagram of local ring homomorphisms
\[
\xymatrixrowsep{2pc}
\xymatrixcolsep{4pc}
\xymatrix{
S \ar@{.>}[r]^\subseteq\ar@{.>>}[d]_{\pi\vert_S} & T \ar@{->>}[d]^{\pi}\\
R \ar[r]^\subseteq & \R \\
}
\]
where $S$ is a regular local ring containing a generating set of $\ker\pi$, whose completion is naturally
isomorphic to $T$, and $\pi\vert_S$ is a surjection. The converse holds if $R$ contains a field of characteristic zero.
\end{introthm}

\begin{proof}
Of course only the last statement of the theorem is not obvious.
To see the last statement, assume that $R$ is isomorphic
to $S'/I'$ where $S'$ is a regular local ring, and $I'$ is an ideal of $S'$ generated by a regular sequence.  
Without loss of generality we can assume that the embedding dimensions of $R$ and $S'$ are the same. Let
$T'$ denote $\widehat{S'}$.
By Cohen's structure theorem for complete local rings of characteristic zero, $T'\cong F[[X_1,\ldots,X_n]]$ where $F\cong T/M$ is a coefficient field.
We can harmlessly identify $T'$ with its isomorphic copy and also write $\widehat R=F[[x_1,\ldots,x_n]]$ where the surjection $T'\to \widehat R$ is the obvious map taking $X_i$ to $x_i$.
We claim there exists an isomorphism
$\varphi:T'\to T$ such that the surjection $T'\to \widehat R$ factors through $\pi$.
Assuming the claim, the subring $S=\varphi(S')$ of $T$ then completes the diagram in the desired fashion.

To see the claim, we first construct a field $E$ such that $E\subseteq F$ is generated over $\mathbf Q$ by a set $A$ of algebraically independent elements and such that $F$ is algebraic over $E$.
Let $B\subset T$ be chosen so that $\pi|_B:B\to A$ is a bijection.
Then $B$ is algebraically independent over $\mathbf Q$ and so $K=\mathbf Q(B)$ is a field contained in $T$ satisfying $\pi|_K:K\to E$ is an isomorphism.
Now let $L$ be the algebraic closure of $K$ in the integral domain $T$.
Then $L$ is certainly a field.
Since $T\to T/M$ factors through $\pi$, $T/M$ is algebraic over $K$ and so $L$ is necessarily a maximal subfield of $T$.
Recalling the proof of Cohen's Theorem, this forces $L$ to be a coefficient field of $T$ and so
we have $T= L[[Y_1,\ldots,Y_n]]$ for some $Y_1,\dots,Y_n\in T$.
We can even choose $Y_1,\ldots,Y_n$ such that $\pi(Y_i)=x_i$ for each $i$.
The last key step is to see that $\pi(L)=F$; for this, it is enough that $\pi(L)\subseteq F$.
Consider any $e\in L$.
Then, using separability,  $e$ satisfies an irreducible monic polynomial $g(x)\in K[Z]$ and we can factor $g(Z)=(Z-e)h(Z)$ with $h(Z)\in L[Z]$ and $h(e)\in L$ a unit.
We have an injection $\widetilde \pi:L\to \widehat R$ given by the composition
$L\to \widehat R\to T/M\to F\to \widehat R$ and we let
$\og(Z), \oh(Z), \oog(Z),\ooh(Z)\in \widehat R[Z]$ denote the images of the respective polynomials under $\pi$ and $\widetilde\pi$ respectively.
It is obvious that $\oog(Z)=\og(Z)$, but we will not know $\ooh(Z)=\oh(Z)$ until we have
actually proved our claim.
Now $0=\og(\pi(e))=\oog(\pi(e))=(\pi(e)-\widetilde\pi(e))\ooh(\pi(e))$.
As $\ooh(\pi(e))-\ooh(\widetilde\pi(e))$ is divisible by the non-unit $\pi(e)-\widetilde\pi(e)$,
$\ooh(\pi(e))$ is a unit and it follows that $\pi(e)-\widetilde\pi(e)=0$ and indeed $\pi(e)\in F$.
Finally there is a unique homomorphism $\varphi:T'\to T$ such that $\varphi|_F$ is the inverse of $\pi|_L$
and $\varphi(X_i)=Y_i$ and this map is what we need to complete the proof of the claim.
\end{proof}

\section{One Dimensional complete intersection domains are complete intersection domains}

In this section we prove the diagram of the theorem in the introduction
can always be completed provided $R$ is a one dimensional integral domain.
The following result from \cite{sing} will help us find such an $S$.

\begin{prop}\label{prop1} $($\cite[Proposition 1]{sing}$)$
Let $(S,\m\cap S)$ be a quasi-local subring of a complete local ring
$(T,\m)$.  Then $S$ is Noetherian and the natural map
$\widehat S \to T$ is an isomorphism if and only if
$S\to T/\m^2$ is onto and $IT\cap S=I$ for every finitely
generated ideal $I$ of $R$.
\end{prop}

We will refer to the condition that $IT\cap S=I$ for every finitely generated ideal $I$ of $R$ by saying that
finitely generated ideals of $R$ are \emph{closed} (with respect to $T$).

\begin{lemma}\label{lemma2}
Let $(T,\m)$ be a local ring and let $f_1,\dots,f_k,s_1,\ldots,s_n$ be a regular sequence in $T$.
Set $K=(f_1,\dots,f_k)T$ and $B=T/K^{i+1}$.
Let $\{\sigma_j\}$ be the set of all distinct monomials of degree $i$ in the generators of $K$,
and let $\overline{\sigma}_j, \overline{s}_i$ denote the respective images in $B$.
If we have an equation $\sum{\overline{a}_j\overline{\sigma}_j}+\sum{\overline{b}_i\overline{s}_i}=0$ in $B$, for $a_j, b_i\in T$, then for each $j$, there exists $\alpha_j\in T$ with $\overline{\alpha}_j\in (\overline{s}_i\ldots,\overline{s}_n)B$ such that $\overline{\alpha}_j\overline{\sigma}_j=\overline{a}_j\overline{\sigma}_j$.
\end{lemma}

\begin{proof}
The equation $\sum{\overline{a}_j\overline{\sigma}_j}+\sum{\overline{b}_i\overline{s}_i}=0$ yields $\sum{a_j\sigma_j}+\sum{b_is_i}\in K^{i+1}$.
As $K^{i+1}=(\{\sigma_j\}T)K$, we get an equation $\sum{(a_j+c_j)\sigma_j}+\sum{b_is_i}=0$ with each $c_j\in K$.
Now fix $j$ and write $\sigma_j=\prod{f_i^{e_i}}$.
Every other $\sigma_i$ is necessarily contained in the ideal $(f_1^{e_1+1},\ldots, f_k^{e_k+1})T$ and so
$(\prod{f_i^{e_i}})(a_j+c_j)\in (f_1^{e_1+1},\ldots, f_k^{e_k+1},s_1,\dots,s_n)T$.
As $f_1,\dots,f_k,s_1,\ldots,s_n$ is regular, it is a straightforward demonstration to see that
$a_j+c_j\in (f_1,\dots,f_k,s_1,\ldots,s_n)T$.
Since $c_j\in K$, this gives that $a_j\in (f_1,\dots,f_k,s_1,\ldots,s_n)T=K+(s_1,\ldots,s_n)T$.
Since $(K/K^{i+1})\overline{\sigma}_j=0$ in $B$, the conclusion follows.
\end{proof}

\begin{lemma}
Let $(T,\m)$ be a Cohen-Macaulay complete local ring, let $f_1,\ldots,f_k$ be a regular sequence
contained in $\m^2$, and set $K=(f_1,\ldots , f_k)T$.  Suppose that
$A$ is a local subring of $T/K^i$ whose completion is naturally isomorphic to $T/K^i$, and
$R$ is a local subring of $T/K$ whose completion is naturally isomorphic to $T/K$ such that the
natural map $T/K^i\to T/K$ induces a surjection $A\to R$.  Assume that $B$ is a quasi-local subring of $T/K^{i+1}$
such that the following hold:
\begin{enumerate}
\item The natural map $T/K^{i+1}\to T/K^i$  induces a surjection
$B\to A$,
\item $f_1+K^{i+1},\ldots, f_k+K^{i+1}\in B$, and
\item $\Ker (B \to T/K)=(f_1+K^{i+1},\ldots, f_k+K^{i+1})B$.
\end{enumerate}
Then  $B$ is a local Noetherian ring whose completion is naturally isomorphic to $T/K^{i+1}$.
\end{lemma}

\begin{proof}
Let $\of_j$ denote $f_j+K^{i+1}$ for $1\le j\le k$.
We first want to show that condition (3) above implies the condition
\[
(3') \,\, \Ker (B \to T/K^i)=(\of_1\ldots, \of_k)^iB.
\]
The statement is trivial if $i=1$.
For $i>1$, let $x\in\Ker (B \to T/K^i)$.
Since $x\in\Ker (B \to T/K)=(\of_1,\ldots,\of_k)B$,
we can write $x=\sum b_j\of_j$ for $b_j\in B$.
We then have $\sum b_j\of_j\in K^2/K^{i+1}$ and  it follows that $b_j\in K/K^{i+1}$ for each $j$.
So $b_j\in\Ker (B\to T/K)=(\of_1,\ldots,\of_k)B$.
Let $\{\sigma_j^{(m)}\}$ be the set of images in $B$ of the distinct monomials
of degree $m$ in the $f_1,\ldots,f_k$.
Then we can write $x=\sum b_{2,j}\sigma_j^{(2)}$ where $b_{2,j}\in B$ for all $j$, finishing the proof of the condition if $i=2$.
For $i>2$,   $\sum b_{2,j}\sigma_j^{(2)}\in K^3/K^{i+1}$ and it follows that $b_{2,j}\in\Ker (B \to T/K)$ for each $j$.
This allows us to write
$x=\sum b_{3,j}\sigma_j^{(3)}$ where $b_{3,j}\in B$ for all $j$. Continuing in this way we arrive at
$x=\sum b_{i,j}\sigma_j^{(i)}$ where $b_{i,j}\in B$ for all $j$. That is $x\in (\of_1,\ldots, \of_k)^iB$, as claimed.

Thus for the rest of the proof we replace condition (3) by condition $(3')$, and continue to let
$\of_j$ denote $f_j+K^{i+1}$ for $1\le j\le k$.
Since $T/K^{i+1}$ modulo the square of its maximal ideal is naturally isomorphic to $T/\m^2$,
by Proposition \ref{prop1} we just need to show that the
map $B \to T/\m^2$ is onto and that finitely generated ideals of $B$
are closed in $T/K^{i+1}$.
Since $T/K^i$ is naturally isomorphic to the completion of $A$, Proposition \ref{prop1} yields that
the map $A\to T/(\m^2+K^i)=T/\m^2$ is onto.  It then follows from assumption
(1) that the map $B\to T/\m^2$ is also onto.

We will now establish that finitely generated ideals of
$B$ are closed in $T/K^{i+1}$.
We first show that all ideals of $B$ which contain $(\of_1,\dots, \of_k)^iB$ are closed.
Suppose $I$ is an ideal of $B$ containing $(\of_1,\dots, \of_k)^iB$ and $x\in I(T/K^{i+1})\cap B$.
Since ideals in $A$ are closed, we have $x+(K^i\cap B)\in I+(K^i\cap B)/(K^i\cap B)$.
As $B\to A$ is surjective, there exists $y\in I$ such that $y+(K^i\cap B)=x+(K^i\cap B)$.
Next, as $x=(x-y)+y$, we may reduce to the case $x\in K^i\cap B$.
Since $(\of_1,\dots, \of_k)^iB\subseteq I$, Condition (3') tells us that $x\in I$, and we are done.

Next we show that if $s_1,\ldots,s_n$ are parameters in $B$, then
$I=(s_1,\ldots,s_n)B$ is closed.
Suppose $x\in I(T/K^{i+1})\cap B$.
Let $J=(\of_1,\dots, \of_k)^iB$.
Since $J+I$ is closed, we have $x=y+z$ with $y\in J$ and $z\in I$.
It suffices to show that $y\in I$, and we may reduce to the case where $z=0$.
Write $x=\sum a_ms_m$ for $a_m\in T/K^{i+1}$, and $y=\sum b_j\sigma_j^{(i)}$ with $b_j\in T/K^{i+1}$. In $T/K^{i+1}$ we have the
equation $\sum a_ms_m-\sum b_j\sigma_j^{(i)}=0$.
Using Lemma \ref{lemma2}, we may assume $b_j\in (s_1,\dots,s_n)(T/K^{i+1})$ for every $j$, and therefore reduce to the case where $x\in JI(T/K^{i+1})$.
This allows us to write $x=\sum c_j\sigma_j^{(i)}$ with each $c_j\in I(T/K^{i+1})$.
As $J$ is closed, we also can write $x=\sum{d_j\sigma_j^{(i)}}$  with $d_j\in B$.
Hence $\sum{(d_j-c_j)\sigma_j^{(i)}}=0$.
It follows that $d_j-c_j\in K/K^{i+1}$ for each $j$.
For fixed $j$ we have $c_j+K\in R$ because $d_j\in B$, and it is also in the closure of $(s_1+K,\ldots,s_n+K)R$. However, ideals in $R$ are closed and so we have elements $e_1,\dots,e_n\in R$ with
$c_j+K=\sum e_m(s_m+K)$.
We next choose a preimage $e_m^\sharp\in B$ for each $e_m$.
Since $\sigma_j^{(i)}c_j=\sigma_j^{(i)}\sum e_m^\sharp (s_m+K^{i+1})$ for each $j$, we have $x\in I$.

Next we show that any ideal $I$ of $B$ which is primary to the maximal ideal of $B$ is closed.
Since $I$ is primary to the maximal ideal, it necessarily contains an ideal $J$ which is generated by a complete system of parameters. We have just seen that $J$ is closed.
Since $B\to T/\m^2$
is onto, so is the map $B\to T/(\m^j+K^{i+1})$ for all $j\ge 1$ (see, for
example, the proof of Proposition 1 of \cite{sing}).
Let $J^\sharp$ denote an ideal of $T$ generated by a set of preimages in $T$ of a generating set for $J$.  Now for some $j$
we have $\m^j\subseteq (K^{i+1}+J^\sharp)T$, so that the map $B\to T/(K^{i+1}+J^\sharp)T$ is also onto.
Thus $T/K^{i+1}=J(T/K^{i+1})+B$, and we have $I(T/K^{i+1})\cap B= (IJ(T/K^{i+1})+I)\cap B\subseteq J(T/K^{i+1})\cap B+I=J+I=I$ as desired.

According to the remark following Lemma 21 in \cite{sing}, if the finitely
generated ideals in $B$ are not all closed, then either there exists
an ideal $I$ of $B$ which is not closed and whose closure is primary to the maximal ideal, or there exists an infinitely generated prime ideal $P\cap B$ for $P\in \Spec T/K^{i+1}$.
To see that neither of these situations can occur, first note that since both of the maps $B\to A$ and
$T/K^{i+1}\to T/K^i$ have nilpotent kernels, there are natural bijections $\Spec B \leftrightarrow \Spec A$ and $\Spec T/K^{i+1} \leftrightarrow \Spec T/K^i$.  Suppose that $I$ is an ideal of $B$ which is not
primary to the maximal ideal of $B$.  Thus $I$ is contained in a non-maximal prime ideal of $B$, and
one of our bijections tells us the same is true for the image $\overline I$ of $I$ in $A$.
As $T/K^i$ is the completion of $A$, this means that $\overline I(T/K^i)$ is contained in a non-maximal prime ideal of $T/K^i$.
Then we invoke the other bijection to see that $I(T/K^{i+1})$ is not primary to the maximal ideal of $T/K^{i+1}$ and so the closure $I(T/K^{i+1})\cap B$ of $I$ is not primary to the maximal ideal of $B$.
Thus if $I$ is an ideal of $B$ whose closure is primary to the maximal ideal of $B$, then $I$ is primary to the maximal ideal of $B$, and hence is closed by the proof above.
 For the other case, $P\cap B$ is necessarily the closure of a finitely generated ideal $J$ of $B$, since all ideals of $T/K^{i+1}$ are finitely generated. As each $f_j$ becomes nilpotent in $T/K^{i+1}$, we see that $(f_1,\ldots,f_k)\subseteq P\cap B$.  Therefore $P \cap B$ is also the closure of the finitely generated ideal $J+(f_1,\ldots,f_k)$, which is closed by the proof above.
This forces $P\cap B$ to be finitely generated.
\end{proof}

\begin{lemma}\label{lemma4}
Let $(T,\m)$ be a Cohen-Macaulay complete local ring, $f_1,\ldots,f_k$ be a $T$-regular sequence
contained in $\m^2$, and $K=(f_1,\ldots , f_k)T$.
Further suppose that $T/K$ has dimension one.
Let $R_i$ and $R$ be local subrings of $T/K^i$ and $T/K$, respectively, with $i\geq 1$ such that
\begin{enumerate}
\item $R$ is an integral domain,
\item $T/K^i$ and $T/K$ are naturally isomorphic to the completions of $R_i$ and $R$ respectively,
\item $f_1+K^i,\ldots, f_k+K^i$ are in $R_i$, and
\item We have the following commutative diagram with surjective vertical maps

\[
\xymatrixrowsep{2pc}
\xymatrixcolsep{2pc}
\xymatrix{
 & T/K^{i+1} \ar@{->>}[d] \\
R_i \ar@{->>}[d] \ar[r]^\subseteq & T/K^{i} \ar@{->>}[d]  \\
R \ar[r]^\subseteq & T/K
}
\]
\end{enumerate}
Then there exists a local subring $R_{i+1}$ of $T/K^{i+1}$ such that $T/K^{i+1}$ is naturally isomorphic to the completion of $R_{i+1}$, $f_1+K^{i+1},\ldots,f_k+K^{i+1}$ are in $R_{i+1}$, and the commutative diagram above may be completed to one with surjective vertical maps
\[
\xymatrixrowsep{2pc}
\xymatrixcolsep{2pc}
\xymatrix{
 R_{i+1} \ar@{.>>}[d] \ar[r]^\subseteq & T/K^{i+1} \ar@{->>}[d]\\
R_i \ar@{->>}[d] \ar[r]^\subseteq & T/K^{i} \ar@{->>}[d]\\
R \ar[r]^\subseteq & T/K
}
\]
\end{lemma}

\begin{proof}
If we construct $R_{i+1}$ containing $f_1+K^{i+1},\ldots,f_k+K^{i+1}$ such that the upper square is commutative with surjective vertical maps and
$\Ker (R_{i+1}\to T/K) = (f_1+K^{i+1},\ldots,f_k+K^{i+1})R_{i+1}$, then the previous lemma completes the proof.
We let $\overline f_j$ denote $f_j+K^{i+1}$ for $1\le j\le k$.
\medskip

Consider the set of subrings $B$ of $T/K^{i+1}$ satisfying the following conditions:
\begin{enumerate}
\item \begin{enumerate}
  \item $\overline f_1,\ldots, \overline f_k\in B$
  \item In addition, if $\Char (T/K)=p>0$, $C^p\subseteq B$ where $C$ is the full preimage of $R_i$ in $T/K^{i+1}$.
\end{enumerate}
\item The image of $B$ under the map $T/K^{i+1}\to T/K^i$ is contained in $R_i$.
\item $\Ker (B\to T/K)  \subseteq (\overline f_1,\ldots, \overline f_k)\pi^{-1}(R_i)$, where $\pi:T/K^{i+1}\to T/K^i$ is the natural projection.
\end{enumerate}
This set can be ordered by inclusion.
We claim that this set contains a maximal element, a claim we will prove by Zorn's Lemma.
To see the claim, first we must show that the set is nonempty.
In the characteristic zero case, let $B_0=\mathbf{Z}[\overline f_1,\ldots, \overline f_k]$.
Obviously $B_0$ satisfies Conditions (1) and (2).
As $\Ker (B_0\to T/K) = (\overline f_1,\ldots, \overline f_k)B_0$, Condition (3) also holds.

In the characteristic $p$ case,
let $B_0=C^p[\overline f_1,\ldots, \overline f_k]$.
Conditions (1) and (2) are clear for $B_0$.
Suppose $\alpha\in\Ker (B_0\to T/K)$.
Then $\alpha=\sum c_j\sigma_j$ with each $c_j\in C^p$ and $\{\sigma_j\}$ ranging
over a set of distinct monomials in the $\overline f_j$.
To show Condition (3) for $B_0$, it suffices to show $ c_j\sigma_j\in (\overline f_1,\ldots, \overline f_k)\pi^{-1}(R_i)$ for each $j$.
The statement is obvious unless $\sigma_j=1$, so we may assume $\alpha\in C^p$.
Suppose $a+K^{i+1}\in C$ is such that $a^p+K^{i+1}\in K/K^{i+1}$.
Then the induced map $C\to R$ takes $a+K^{i+1}$ to a nilpotent element of the integral domain $R$.
So $a\in K$.
Hence $a+K^i=(a_1+K^i)(f_1+K^i)+\cdots+(a_k+K^i)(f_k+K^i)$ with each $a_i\in T$.
As $T/K^i$ is the completion of $R_i$, we may actually choose $a_1,\ldots, a_k$ so that each  $a_j+K^i$ is in $R_i$; so $a_j+K^{i+1}\in \pi^{-1}(R_i)$.
Let $e=\sum a_jf_j$ and note that $a+K^{i+1}=(e+K^{i+1})+(d+K^{i+1})$ with $d\in K^i$.
Finally, since $d^2\in K^{i+1}$ and $pd\in K^{i+1}$,
$(a+K^{i+1})^p=(e+K^{i+1})^p\in (\overline f_1,\ldots, \overline f_k)\pi^{-1}(R_i)$.

We have thus shown that in any characteristic, the set is nonempty.
Next we consider the union of an ascending chain of elements in the set.
The union obviously satisfies Condition (1) and the other two conditions can be viewed as elementwise conditions.
Since these hold for every set in the union, they must hold for the union.
Thus, by Zorn's Lemma, the set contains a maximal member $B$.

Next we claim the map $B\to R_i$ is surjective.
If not, let $A_i = \Image (B\to R_i)$ and let $A = \Image (B\to R)$.
We choose $r^{\sharp}\in R_i-A_i$ and let $r$ be the image of $r^{\sharp}$ in $R$.
If we lift $r^{\sharp}$ to a preimage $\widetilde{r}\in T/K^{i+1}$, we have natural surjections
$B[\widetilde r]\to A_i[r^\sharp] \to A[r]$ fitting into the diagram
\[
\xymatrixrowsep{2pc}
\xymatrixcolsep{2pc}
\xymatrix{
B \ar[r]^\subsetneq \ar@{->>}[d] & B[\widetilde r] \ar@{->>}[d] \ar[rr]^\subseteq & & T/K^{i+1} \ar@{->>}[d]\\
A_i \ar[r]^\subsetneq \ar@{->>}[d] & A_i[r^\sharp] \ar@{->>}[d] \ar[r]^\subseteq & R_i \ar@{->>}[d] \ar[r]^\subseteq & T/K^{i} \ar@{->>}[d]\\
A \ar[r]^\subseteq & A[r] \ar[r]^\subseteq & R \ar[r]^\subseteq & T/K
}
\]
where since $A_i[r^{\sharp}]$ properly contains $A_i$, $B[\widetilde{r}]$ properly contains $B$.
We will show that $\widetilde{r}$ can be chosen in such a way that $B[\widetilde{r}]$
satisfies the three conditions.
This will contradict the maximality of $B$ and so prove the claim.
As the first two conditions hold for an arbitrary choice of $\widetilde{r}$,
we need only consider Condition (3).

Note that $\Ker (B[\widetilde{r}]\to T/K) = \Ker (B[\widetilde{r}]\to A[r])$ since $R$ injects into $T/K$. We have a natural presentation
\[
A[X]/I \xrightarrow{\cong} A[r]
\]
For $h(X)\in B[X]$ we let $\overline h(X)$ denote the polynomial in $A_i[X]$ obtained
by reducing the coefficients of $h(X)$ modulo $K^i$, and by $\overline{\overline h}(X)$
the polynomial in $A[X]$ obtained by reducing the coefficients of $h(X)$ modulo $K$.

If $h(X)\in B[X]$, then $h(\widetilde{r})\in \Ker (B[\widetilde{r}]\to T/K)$ precisely if $\ooh(X)\in I$.
So the proof reduces to showing that $h(\widetilde{r})\in (\overline f_1,\ldots, \overline f_k)\pi^{-1}(R_i)$ whenever $\ooh(X)\in I$.
First we consider the case $\ooh(X)=0$.
By Condition (3) for $B$, all of the coefficients of $h(X)$ are in $(\overline f_1,\ldots, \overline f_k)\pi^{-1}(R_i)$
and so $h(\widetilde{r})\in (\overline f_1,\ldots, \overline f_k)\pi^{-1}(R_i)$ regardless of which lifting $\rt$ we choose.
In particular, we note that Condition (3) holds for $B[\widetilde{r}]$ if $I=(0)$.

Assume $I\neq (0)$.
Suppose $g(X)\in B[X]$ is such that $0\neq \oog(X)\in I$ and
$g(X)$ is of minimal degree among all such polynomials.
As $B\to A$ is surjective, every element of $I$ has the form $\ooh(X)$
and so $\oog(X)$ also has minimal degree among the set of nonzero polynomials in $I$.
Since $R$ is an integral domain, $\oog(X)$ is a (not necessarily monic) minimal polynomial satisfied by $r$ over the quotient field of $A$.

\medskip

\noindent\emph{Claim.}
We may choose $g(X)$ and $\rt$ so that $g(\rt)\in (\of_1,\ldots,\of_k)\pi^{-1}(R_i)$.

\medskip

Either $\oog^{\prime}(r)=0$ or $\oog^{\prime}(r)\neq 0$.
The first case will occur precisely when $A\subseteq A[r]$ is not a separable extension, something which can happen only if $T/K$ has characteristic $p$ and $g(X)$ has degree at least $p$.
In this case $(\rt)^p\in B$, so we may choose $g(X)=X^p-(\rt)^p$ and we see that $g(\rt)=0\in (\of_1,\ldots,\of_k)\pi^{-1}(R_i)$.
The claim actually holds for any choice of $\rt$.

In the second case, we choose our minimal polynomial $g(X)$ arbitrarily but we must select $\rt$ carefully.
Since $g(\rt)\in K/K^{i+1}$, $\og(r^{\sharp})\in (\pi(\of_1),\ldots,\pi(\of_k))(T/K^i)$.
Further, as $T/K^i$ is the completion of $R_i$, $\og(r^{\sharp})\in (\pi(\of_1),\ldots,\pi(\of_k))R_i$.
 So we may choose elements $\gamma_j\in\pi^{-1}(R_i)$ such that $\og(r^{\sharp})=\sum {\pi(\of_j)\pi(\gamma_j)}$.
 Let $\{\sigma_j^{(m)}\}$ be the images in $B$ of the distinct monomials of degree $m$ in the
 $f_1,\ldots,f_k$.
 Then we have elements $\alpha_j\in T/K^{i+1}$ such that
 $g(\rt)=\sum {\of_j\gamma_j}+\sum{\sigma_j^{(i)}\alpha_j}$.
 As $\oog^{\prime}(r)\neq 0$, $R/(\oog^{\prime}(r))R$ is zero-dimensional and so $R_i/(\og^{\prime}(r^{\sharp}))R_i$ is also zero-dimensional.
Since zero-dimensional local rings are complete,
the induced map
\[
R_i/(\og^{\prime}(r^{\sharp}))R_i\to (T/K^i)/(\og^{\prime}(r^{\sharp})(T/K^i))
\]
is an isomorphism.
It follows that $T/K^i=R_i+\og^{\prime}(r^{\sharp}) (T/K^i)$ and by taking preimages we see
$T/K^{i+1}=\pi^{-1}(R_i)+g^{\prime}(\rt) (T/K^{i+1})$.
There exists for each $j$, $\beta_j\in \pi^{-1}(R_i),t_j\in T/K^{i+1}$ such that
$\alpha_j=\beta_j+g^{\prime}(\rt) t_j$.
We now claim that the desired condition holds if we choose the lifting $\rt_2=\rt-\sum {\sigma_j^{(i)}t_j}$.
Taking the Taylor series expansion of $g(X)$ about
$\rt$ we have
\begin{align*}
g(\rt_2)&=g(\rt)-g'(\rt)\sum {\sigma_j^{(i)}t_j}\\
&=\sum {\of_j\gamma_j}+\sum\sigma_j^{(i)}(\alpha_j-t_jg'(\rt))\\
&=\sum {\of_j\gamma_j}+\sum\sigma_j^{(i)}\beta_j\in (\of_1,\ldots,\of_k)\pi^{-1}(R_i).
\end{align*}
This completes the proof of the claim.

To derive our contradiction, it only remains to show that, for the choice of $\rt$ given by this claim,
if $h(X)\in B[X]$ with $\ooh(X)\in I$, then $h(\rt)\in (\of_1,\ldots,\of_k)\pi^{-1}(R_i)$.
Choosing $b\in B-K/K^{i+1}$ to be a sufficiently high power of the leading coefficient of $g(X)$,
we can write $bh(X)=h_1(X)g(X)+h_2(X)$ where $h_1(X),h_2(X)\in B[X]$ and $h_2(X)$ is a polynomial of lower degree than $g(X)$.
Since $\ooh(X)\in I$ and $\oog(X)\in I$, we have $\ooh_2(X)\in I$.
By degree considerations $\ooh_2(X)=0$ and, as we noted above, this yields
$h_2(\rt)\in (\overline f_1,\ldots, \overline f_k)\pi^{-1}(R_i)$.
Further, as $h_1(\rt)\in \pi^{-1}(R_i)$ and $g(\rt)\in (\of_1,\ldots,\of_k)\pi^{-1}(R_i)$,
we actually get $bh(\rt)\in (\of_1,\ldots,\of_k)\pi^{-1}(R_i)$.
Additionally, we note that since $b$ maps to a nonzero element of $R$ and every nonzero element of $R$ is regular on the completion of $R$, i.e., $T/K$, $b=(e+K^{i+1})$ where $f_1,\ldots,f_k,e$ is a regular sequence in $T$.
It follows that $b$ is a regular element on $T/K^m$ for every $m\leq i+1$.

To get $h(\rt)\in (\of_1,\ldots,\of_k)\pi^{-1}(R_i)$, we will actually prove a more general statement which allows the coefficients of $h(X)$ to be arbitrary elements of $T/K^{i+1}$.
Using reverse induction on $m$, for $m=1,\ldots,i$, we will show that
if $h(X)\in (T/K^{i+1})[X]$, $b\in B-K/K^{i+1}$, and $bh(\rt)\in (\of_1,\ldots,\of_k)^m\pi^{-1}(R_i)$, then $h(\rt)\in (\of_1,\ldots,\of_k)^m\pi^{-1}(R_i)$.
First consider $m=i$.
Here we have $bh(\rt)=\sum \sigma_{j}^{(i)}d_j$ with each $d_j\in \pi^{-1}(R_i)$.
Also, as $b$ is regular on $T/K^i$, we have $h(\rt)=\sum \sigma_{j}^{(i)}c_j$ with each $c_j\in T/K^{i+1}$.
It follows that $\sum \sigma_{j}^{(i)}(bc_j-d_j)=0$ and so $bc_j-d_j\in K/K^{i+1}$.
Thus $\pi(d_j)\in (\pi(b),\pi(\of_1),\dots,\pi(\of_k))(T/K^i)$.
As $T/K^i$ is the completion of $R_i$ and so ideals are closed, we get
$\pi(d_j)\in (\pi(b),\pi(\of_1),\dots,\pi(\of_k))R_i$.
Thus there exists $a_j\in \pi^{-1}(R_i)$ such that $\pi(d_j-ba_j)\in K/K^i$.
Of course, we then have $d_j-ba_j\in K/K^{i+1}$ and so $\sigma_{j}^{(i)}(d_j-ba_j)=0$.
So $bh(\rt)=\sum \sigma_{j}^{(i)}ba_j$, giving $h(\rt)=\sum \sigma_{j}^{(i)}a_j\in (\of_1,\ldots,\of_k)^i\pi^{-1}(R_i)$.
Next assume that $m\geq 1$ and that we have already demonstrated the $m+1$ case.
As before, we have $bh(\rt)=\sum \sigma_{j}^{(m)}d_j$ with each $d_j\in \pi^{-1}(R_i)$ and
 $h(\rt)=\sum \sigma_{j}^{(m)}c_j$ with each $c_j\in T/K^{i+1}$.
 Again this gives $\sum \sigma_{j}^{(m)}(bc_j-d_j)=0$ and so $bc_j-d_j\in K/K^{i+1}$.
As before, this gives us an element $a_j\in \pi^{-1}(R_i)$ such that
$\pi(d_j-ba_j)\in (\pi(\of_1),\ldots,\pi(\of_k))R_i$.
Thus $d_j-ba_j\in (\of_1,\dots,\of_k)\pi^{-1}(R_i)+K^i/K^{i+1}$.
Then $\sigma_{j}^{(m)}(d_j-ba_j)\in (\of_1,\dots,\of_k)^{m+1}\pi^{-1}(R_i)$.
Finally we define $k(X)=h(X)-\sum \sigma_{j}^{(m)}a_j$.
Then $bk(\rt)=\sum \sigma_{j}^{(m)}d_j-b\sum \sigma_{j}^{(m)}a_j=\sum \sigma_{j}^{(m)}(d_j-ba_j)\in
(\of_1,\dots,\of_k)^{m+1}\pi^{-1}(R_i)$.
By the induction assumption, $k(\rt)\in (\of_1,\ldots,\of_k)^{m+1}\pi^{-1}(R_i)$
and so $h(\rt)\in (\of_1,\ldots,\of_k)^m\pi^{-1}(R_i)$ as desired.
The $m=1$ case is actually the result we need, the final piece in the demonstration of our contradiction.

We have shown that we can choose $B$ satisfying the three conditions such that $B\to R_i$ is surjective. Let $R_{i+1}=B$.
Condition (3) gives $\Ker (R_{i+1}\to T/K)  \subseteq (\overline f_1,\ldots, \overline f_k)\pi^{-1}(R_i)$.
As $R_{i+1}\to R_i$ is surjective, we now have
$\Ker (R_{i+1}\to T/K)  \subseteq (\overline f_1,\ldots, \overline f_k)(R_{i+1}+\Ker (T/K^{i+1}\to T/K^i))=
(\overline f_1,\ldots, \overline f_k)R_{i+1}$.
\end{proof}

\begin{thm}\label{thm1}
Let $(T,\m)$ be a Cohen-Macaulay complete local ring, $f_1,\ldots,f_k$ be a $T$-regular sequence
contained in $\m^2$, and $K=(f_1,\ldots , f_k)T$. Further suppose that $T/K$ has dimension one.
Let $R$ be an integral domain which is a local subring of  $T/K$ such that $T/K$ is naturally isomorphic
to the completion of $R$.
Then there exists a local subring $S$ of $T$ with $ f_1,\ldots,  f_k\in S$ such that $T$ is naturally isomorphic to the completion of $S$ and we have the commutative diagram with surjective vertical maps
\[
\xymatrixrowsep{2pc}
\xymatrixcolsep{4pc}
\xymatrix{
S \ar@{.>}[r]^\subseteq\ar@{.>>}[d]_{\pi\vert_S} & T \ar@{->>}[d]^{\pi}\\
R \ar[r]^\subseteq & T/K \\
}
\]
\end{thm}

\begin{proof} We apply the previous lemma a countable number of times to obtain a sequence of surjections
\[
\cdots \to R_{i+1} \xrightarrow{\rho_{i+1}} R_i \xrightarrow{\rho_i} R_{i-1} \to \cdots \to R_1 \xrightarrow{=} R
\]
where $\pi_i:T/K^{i+1}\to T/K^i$ is the natural map and $\rho_i=\pi_i|_{R_i}$. This sequence of maps
forms an inverse system and the inclusion maps $R_i \to T/K^i$ define a morphism of inverse systems. Let $S=\lim_{\leftarrow}R_i$ be the inverse limit of the $R_i$. Elements of $s\in S$ have the form
$s=(s_i+K^i)\in \prod R_i$ with $s_i\in T$ and $s_{i+1} +K^i= s_i +K^i$ for all $i\ge 1$.
Since $T$ is complete in the $K$-adic topology, the natural map $T \to \lim_\leftarrow T/K^i$ taking $t$ to $(t+K^i)$
is an isomorphism. We will therefore identify the element $x$ of $T$ with $(x+K^i)\in \lim_\leftarrow T/K^i$.
It follows that we get a commutative diagram with surjective vertical maps
\[
\xymatrixrowsep{2pc}
\xymatrixcolsep{4pc}
\xymatrix{
S \ar@{.>}[r]^\psi\ar@{->>}[d]_{\pi\vert_S} & T \ar@{->>}[d]^{\pi}\\
R \ar[r]^\subseteq & T/K \\
}
\]

First we show that the induced map on direct limits $\psi$ is injective.  Suppose that $\psi(s)=0$ in $T$
for $s=(s_i+K^i)\in S$.  This means that $s_i+K^i=K^i$ for all $i$, and so $(s_i+K^i)=(K^i)$ which is the zero element in $S$.

Since $f_j + K^i \in R_i$ for all $i$, and $1\le j\le k$. We have $f_j=(f_j+K^i)\in S$ for $1\le j\le k$. It follows that we also get all monomials comprised of the $f_j$ in $S$.

To show that $S$ is quasi-local with maximal ideal $\m\cap S$, we show that every element of $S-\m$ has an inverse in $S$.  Let $x=(x_i+K^i)\in S-\m$.  Since each $R_i$ is quasi-local with maximal ideal $\m/K^i\cap R_i$, and $x_i+K^i\in R_i-(\m/K^i\cap R_i)$, we have inverses $y_i+K^i\in R_i-(\m/K^i\cap R_i)$ of $x_i+K^i$ in $R_i$ for each $i$.  We just need to show that
$y_i+K^{i-1}=y_{i-1}+K^{i-1}$ for all $i\ge 1$, for then the element $(y_i+K^i)$ will be the inverse of $x$ in $S-\m$.
We have $x_iy_i-1\in K^i$, and it follows that $x_{i-1}y_i-1\in K^{i-1}$.  By uniqueness of inverses
in $R_{i-1}$ we see that $y_i+K^{i-1}=y_{i-1}+K^{i-1}$.

If we show that $S\to T/\m^2$ is onto and $IT\cap S = I$ for all finitely generated ideals $I$ of $S$,
then we may apply Proposition 1 to complete the proof.  As $R\to T/\m^2$ is onto, certainly
$S\to T/\m^2$ is onto.

Now let $y_1,\dots,y_n$ be elements of $S$, with $y_j=(y_{j,i}+K^i)$, and $I$ the ideal of $S$ they
generate.  Choose $x=(x_i+K^i)\in IT\cap S$. Then $x_i+K^i\in I (T/K^i)\cap R_i$ for all $i\ge 0$.

\medskip

\noindent \emph{Claim.}  There exists a positive integer $N$ and $\{t_{j,i}| i\ge 0, 1\le j\le n\}\subseteq T$ such that
\begin{enumerate}
\item $t_{j,i}+K^{N+i+1}\in R_{N+i+1}$,
\item $t_{j,i+1}+K^i=t_{j,i}+K^i$, and
\item $x_{N+i}+K^{N+i}=\sum t_{j,i}y_{j,N+i}+K^{N+i}$ in $R_{N+i}$
\end{enumerate}
\medskip
for each $j$ and for all $i\ge 0$. Assuming the claim, conditions (1) and (2) imply that $(t_{j,i}+K^i)\in S$.
Condition (3)
implies $x_i+K^i=\sum t_{j,i}y_{j,i} + K^i$ for $i\ge 0$, which means that
$(x_i+K^i)=\sum (t_{j,i}+ K^i)(y_{j,i}+ K^i)\in I$, and thus $I$ is closed.

\medskip

\noindent\emph{Proof of the claim.} By Artin-Rees there exist an integer $N$ such that
$IT\cap K^{N+i}=K^i(K^N\cap IT)$ for all $i\ge 0$.
We prove the existence of the $t_{j,i}$ by induction on $i$, beginning with $i=0$.

Since finitely generated ideals in $R_{N+1}$ are closed we have  $I(T/K^{N+1})\cap R_{N+1}=IR_{N+1}$.
Therefore there exists $\{t_{j,0}\}\subseteq T$ such that
\[
x_{N+1} + K^{N+1}= \sum t_{j,0}y_{j,N+1} + K^{N+1}
\]
with $t_{j,0}+K^{N+1}\in R_{N+1}$.  Of course then,
\[
x_N + K^N= \sum t_{j,0}y_{j,N} + K^N.
\]

Now suppose we have successfully chosen $\{t_{j,i}| 1\le j\le n\}$ and we want to find
$\{t_{j,i+1}| 1\le j\le n\}$.  Since finitely generated ideals of $R_{N+i+1}$ are closed, there exists
$\{u_{j,i}\}\subseteq T$ such that
\[
x_{N+i+1}+ K^{N+i+1}=\sum u_{j,i}y_{j,N+i+1}+K^{N+i+1}
\]
and $u_{j,i}+K^{N+i+1}\in R_{N+i+1}$.
This equation together with
\[
x_{N+i+1}+K^{N+i}=\sum t_{j,i}y_{j,N+i+1}+K^{N+i}
\]
yields
\begin{align*}
\sum (u_{j,i}-t_{j,i})y_{j,N+i+1}\in K^{N+i}\cap (IT+K^{N+i+1})&=(K^{N+i}\cap IT)+K^{N+i+1}\\
&=K^i(K^N\cap IT)+K^{N+i+1}\\
&\subseteq K^iIT + K^{N+i+1}
\end{align*}
Therefore we have $\sum (u_{j,i}-t_{j,i})y_{j,N+i+1}+K^{N+i+1}\in K^iI(T/K^{N+i+1})\cap R_{N+i+1}$.
Since finitely generated ideals of $R_{N+i+1}$ are closed, there exist $\{v_{j,i}\}\subseteq T$ such that
$v_{j,i}+K^{N+i+1}\in K^i(T/K^{N+i+1})\cap R_{N+i+1}$ and
\[
\sum(u_{j,i}-t_{j,i})y_{j,N+i+1}+K^{N+i+1}=\sum v_{j,i}y_{j,N+i+1}+K^{N+i+1}
\]
Finally, since $R_{N+i+2}\to R_{N+i+1}$ is onto, we choose
$t_{j,i+1}+K^{N+i+2}\in R_{N+i+2}$ such that $t_{j,i+1}+K^{N+i+1}=t_{j,i}+v_{j,i}+K^{N+i+1}$.  It is routine to check that the elements $\{t_{j,i+1}\}$ do the job and the induction is complete.
\end{proof}

Assume that $\R$ is isomorphic to $T/K$.  Upon identifying $R$ with its image in $T/K$,
Theorem \ref{thm1} proves the following.

\begin{cor} Suppose that $R$ is a local Noetherian integral domain of
dimension 1.  Then $\R$ is isomorphic to the quotient of a regular local ring by an ideal generated
by a regular sequence if and only if $R$ is the quotient of a regular local ring by an ideal generated
by a regular sequence.
\end{cor}

% Section 2 begins here
\section{A 3-dimensional complete intersection that is not a complete intersection}

This section is devoted solely to the construction of an example.
The ring $R$ will be a three dimensional local domain whose completion is $T=\mathbf{R}[[x,y,z,w]]/(x^2+y^2)$.
The example is far from excellent; in fact, while we do not know whether or not there are excellent examples, the non-excellence is a critical aspect of this example.
The generic point of the singular locus of $T$, the prime ideal $(x,y)T$, will intersect trivially with $R$, but $R$ will have uncountably many height two prime ideals $P_{\lambda}$ which are contractions of singular points and in fact singular themselves.
Showing that $R$ cannot be lifted to a regular local ring will largely consist of showing that there is not a simultaneous
compatible lifting of the $R/P_{\lambda}$'s.
This particular trick certainly requires both non-excellence and dimension at least three.
On the other hand, there is no reason to believe that the choice of coefficient field was of any importance.
The construction is easier if one knows the cardinality of the coefficient field and choosing $\mathbf{R}$
allowed us to use the very elementary polynomial $x^2+y^2$.

While the construction is rather intricate, the conception behind it is not.
We build an example with the property that if $R$ can be lifted to a regular local ring $S$
contained in $\mathbf{R}[[x,y,z,w]]$, then $S$ must contain an element
$\Theta=f(1+z\omega+z^2h+xa+yb)$ where $f=x^2+y^2$ and $\omega\in\mathbf{R}$ is known, but we have no information about $h,a,b$.
We also equip $R$ with a large collection of prime ideals whose extension to $T$ contains $(x,y)T$ and the construction of each of these prime ideals guarantees the existence of an element in the lifting which is congruent to $f$ modulo that particular prime ideal $P$.
This new element and $\Theta$ are unit multiples of each other and so their quotient will
be in $S$.
So $1+z\omega+z^2h\in R/P$ and it follows that $z(\omega+zh)\in R/P$.
We will construct $R$ so that $z\in R$ and so $\omega+zh$ is in the quotient field of $R/P$.
We can't control $h$ but we do know that, for some fixed value of $h$, $\omega+zh$ must be in the quotient field of $R/P$ for every one of the primes we construct.
So, for every possible value of $h$, we construct a prime ideal $P_h$ such that $\omega+zh$
is transcendental over $R/P_h$ and so certainly not in the quotient field.
This contradiction rules out the possibility of a lifting.

We begin with a lemma.
It is a minor variant of Lemma 3 from \cite {UFD} and the proof is substantively the same.

\begin{lemma}\label{avoidance}
Let $(T,\m)$ be a local ring which contains an uncountable field $F$,
 let $C\subset \Spec T$, and let $D\subset T$.
Let $I$ be an ideal of $T$ with $I\nsubseteq P$ for all $P\in C$.
If $|C\times D|< |F|$, then $I\nsubseteq \bigcup \{r+P|  r\in D, P\in C\}$.

Moreover, if $t\in I-\bigcup_{P\in C}P$ (and such a $t$ must exist),
then there exists $u\in F$ such that $tu\notin \bigcup \{r+P| r\in D, P\in C \}$.
Further, if $D^{\prime}\subset T$  is such that
$|D^{\prime}|< |F|$, then we may additionally choose $u\notin \bigcup \{r+\m|  r\in D^{\prime}\}$.

\end{lemma}

\begin{proof}
We first prove $I\nsubseteq \bigcup\{P| P\in C\}$.
It suffices to prove $I\nsubseteq \bigcup\{P+\m I | P\in C\}$.
By Nakayama's Lemma, $I\nsubseteq P+\m I$ for any $P\in C$.
Letting $V=I/\m I$, we have reduced the problem to showing that if a finite-dimensional
vector space $V$ over $T/\m$ is the union of $|C|$ subspaces where $|C|<|T/\m|$, then one
subspace must be all of $V$.
This is surely well known and the proof is even written out as part of the proof of Lemma 3 in \cite{UFD}.

To prove the lemma, it suffices to prove the second paragraph.
We choose an element $t\in I-\bigcup\{P | P\in C\}$.
Then, for any $(r,P)\in D \times C$, 
we have $tu\in r+P$ if and only if $u\equiv t^{-1}r$ modulo $P$.
If $r+P\notin (t+P)(T/P)$, then $tu\notin r+P$.
Otherwise, $t^{-1}r\equiv s$ modulo $P$ and we can obtain $tu\notin r+P$ by choosing $u\notin s+P$.
For each such pair $(r,P)$, we choose a coset representative $s$ and we label the full set of
coset representatives $D^*$.
Then $|D^*\cup D^{\prime}|\leq |C \times D|+|D^{\prime}|<|F|$ and so we can choose $u\in F$ such that $u\not\equiv s$ modulo $\m$ for any $s\in D^*\cup D^{\prime}$.
Clearly $u\notin \bigcup \{s+P| s\in D^*, P\in C\}\bigcup \{r+\m|  r\in D^{\prime}\}$
and so $tu\notin \bigcup \{r+P|  r\in D, P\in C\}$ and $u\notin \bigcup \{r+\m|  r\in D^{\prime}\}$.

\end{proof}

We repeat a definition from \cite{UFD}.

\begin{defn}
Let $(T,\m)$ be a complete local ring and let $(R,\m\cap R)$ be a quasi-local unique factorization domain contained in $T$ satisfying:
\begin{enumerate}
\item $|R| \leq \sup (\aleph_0,|T/\m|)$ with equality only if $|T/\m|$ is countable,
\item $Q\cap R=(0)$ for all $Q\in \Ass (T)$, and
\item if $t\in T$ is regular and $P\in \Ass (T/tT)$, then $\Ht (P\cap R) \leq 1$.
\end{enumerate}
Then $R$ is called an N-\emph{subring} of $T$.
\end{defn}

In the present circumstances, we will let $T=\mathbf{ R}[[x,y,z,w]]/(x^2+y^2)$.
Here condition (2) of N-subring is vacuous and condition (1) is just $|R|<|\mathbf{ R}|$.

\begin{defn}
$(R,\Gamma)$ is called an N-pair if $R$ is an N-subring of $T$ and $\Gamma$ is a set of pairs
$ \{(P_{\lambda},V_{\lambda}) | \lambda\in \Gamma\}$ with $\{P_{\lambda}\}$ a distinct set of primes such that, for every $\lambda$,

\begin{enumerate}
  \item $z,w\in R$
  \item  $(x,y)T\cap R=(0)$
   \item $P_{\lambda}=(x,y,z+u_{\lambda}w)T$ for some $u_{\lambda}\in \mathbf{ R}\cap R$
   \item $V_{\lambda}\subset T$ and $|V_{\lambda}|\leq |R|$
   \item the images in $T/P_\lambda$ of the elements of $V_\lambda$ are algebraically independent over $R/P_\lambda \cap R$
\item if $Q_{\lambda}=P_{\lambda}\cap R$, then
 $Q_{\lambda}R_{Q_{\lambda}}=(x+\alpha_{\lambda}t_{\lambda}, y+\beta_{\lambda}t_{\lambda},t_{\lambda})R_{Q_{\lambda}}$
 for some $\alpha_{\lambda},\beta_{\lambda}\in \mathbf{R}$ where $t_{\lambda}=z+u_{\lambda}w$.

\
\end{enumerate}
Note that the requirement that the elements of the set $\{P_{\lambda}\}$ be distinct and of the specified form forces $|\Gamma|\leq|R|$.
In practice, $V_{\lambda}$ will always be a singleton set, but the additional generality does not require any extra effort.
\end{defn}

\begin{lemma}\label {3}
Suppose  $R$ is an N-subring of $T$, $v\in T$,
$\Lambda=\{(P_{\lambda},V_{\lambda}) |\lambda\in\Lambda\}$ where each $P_{\lambda}\in\Spec T$ and each $V_{\lambda}$ is a subset of $T$, $H$ is a finite subset of the set of height one prime ideals of $T$ which have the form $rT$ for $r\in R$, and $I$ is an ideal of $T$.
Let $G=\{P\in  \Ass (T/rT)| r\in R\}$ and $\Delta=\{P_{\lambda} | \lambda\in\Lambda\}$.
Assume
\begin{enumerate}
  \item $I\nsubseteq P$ for every $P\in \Delta\cup (G-H)$
  \item if $P^{\prime}\in H$, then $P^{\prime}\nsubseteq P$ for every $P\in \Delta$
  \item  $|\Lambda|\leq |R|$ and $|V_{\lambda}|\leq |R|$ for every $\lambda$
\end{enumerate}
Then there exists an element $d\in I$ such that if $S$ is the intersection of $T$ with the quotient field of $R[v+d]$, $S$ is an N-subring of $T$ such that, for each $\lambda\in\Lambda$,
the image of $v+d$ in $T/P_{\lambda}$ is transcendental over $(R/P_{\lambda}\cap R)[\overline{V}_{\lambda}]$ and for each $P\in G-H$, the image of $v+d$ in $T/P$ is transcendental over $(R/P\cap R)$.
Moreover, if $I=tT$ is a principal ideal, we may choose $d=tu$ with $u\in\mathbf{R}$
and $u\notin \bigcup \{r+\m|  r\in D^{\prime}\}$ for any prescribed set $D^{\prime}$ with $|D^{\prime}|<|\mathbf{R}|$.
\end{lemma}

\begin{proof}
Let $C=\Delta\cup (G-H)$.
For each $P\in C$, we define a subring $A(P)$ of $T/P$.
If $P\notin \Delta$, $A(P)=R/P\cap R$.
If $P=P_{\lambda}\in\Delta$, $A(P)=(R/P_{\lambda}\cap R)[\overline{V}_{\lambda}]$.
In either case, $A(P)$ has cardinality at most $|R|$ and so the same is true for the algebraic closure of $A(P)$ in $T/P$.
Let $D_P$ be a complete set of coset representatives modulo $P$ of $\{t\in T | v+t \text{ is algebraic over } A(P)\}$.
Therefore $|D_P|\leq |R|$.
Now $|C|\leq |R|$ and so also, if $D=\bigcup_{P\in C}D_P$ then $|D|\leq|R|$, and finally $|C \times D|\leq
|R|<|\mathbf{R}|$.
Hence we may apply Lemma~\ref {avoidance} to choose $d\in I$ so that
the image of $v+d$ in $T/P$, is transcendental over $A(P)$ for every $P\in C$.
We use this choice of $d$ to define $S$.
The construction immediately forces  the image of $v+d$ in $T/P_{\lambda}$ to be transcendental over $(R/P_{\lambda}\cap R)[\overline{V}_{\lambda}]$  for each $\lambda\in\Lambda$
and the image of $v+d$ in $T/P$ to be transcendental over $(R/P\cap R)$ for each $P\in G-H$.
Additionally, in the case where $I=tT$ is principal, Lemma~\ref {avoidance} allows us to choose
$d=tu$ so that the final statement of the lemma holds.

It only remains to show that $S$ is an N-subring.
Suppose $P\in \Ass (T/tT)$ for some nonzero $t\in T$.
If $P\cap R=(0)$ and $Q=P\cap S$, then $R_{(0)}[v+d]\subset S_Q\subseteq R[v+d]_{(0)}$
and $S_Q$ is either a discrete valuation ring or a field.
In either case, $\Ht (P\cap S) \leq 1$.
If $P\cap R\neq (0)$ then $P\in G$.
If $P\in G-H$ and  $Q=P\cap S$, the fact that $v+d$ is transcendental over $R/(P\cap R)$  tells us that $S_Q$ is the localization of $R_{P\cap R}[v+d]$ at the unique prime ideal minimal over $P\cap R$.
Thus $Q$ is a height one prime ideal of $S$ principally generated by the generator of $P\cap R$.
Finally, if $P\in H$, there exists $r\in R$ such that $P=rT$.
Then $Q=P\cap S\subset rT$ and since $rT\cap S=rS$, we see that $Q$ is also a principal height one prime ideal of $S$.
This shows that condition (3) of N-subring is satisfied and that is the only one of the three enumerated conditions that needs to be checked as $|S|=|R|$.
We also note that the intersection of a local domain with a subfield of its quotient field is always quasi-local.
To prove $S$ is a UFD, we may first invert any known principal prime elements and so reduce to the case where our ring is a localization of $R_{(0)}[v+d]$ and so is of course a UFD.
Thus $S$ is an N-subring of $T$.

\end{proof}

\begin{remark}
As we will repeatedly employ Lemma~\ref{3}, it seems useful to give a quick summary of how it is used.
We start with a ring $R$ (which is always obvious from the context) and must specify $v,\Lambda,H,I$.
The lemma then gives us an element $d$ and a ring $S$ satisfying the properties we need.
\end{remark}

\begin{remark}
The reader may have noted that the proof did not require $H$ to be finite, nor did it require
hypothesis (2).
However, as any nonzero ideal $I$ can be contained in at most finitely many
members of $G$, none of which are contained in $\Delta$, these hypotheses place no
restriction on the situations in which Lemma~\ref{3} can be used.
The hypotheses are critical in the next two lemmas.
\end{remark}

It will be useful if we can describe $S$ more precisely.

\begin{lemma}\label {describe}
Let  $S$ be the N-subring obtained in Lemma~\ref{3} and let $\eta=v+d$.
Let $g\in \m\cap R[\eta]$ be a prime element of  $R[\eta]$.
If $g\in R$, then $g$ is a prime element of $S$.
If $g\notin R$, then there is an element $r\in R$, possibly $r=1$, such that every height one prime ideal of $T$ which contains $r$ is in $H$, and $g/r$ is either a prime element or a unit of $S$.
Consequently, if $H=\{r_1T,\ldots,r_nT\}$ with each $r_i\in R$ and $t=\prod r_i$,  then there exists a domain $L$ such that $R[\eta]\subseteq L\subset R[\eta,t^{-1}]$ and $S=L_{\m\cap L}$.
In particular, if $H=\emptyset$, we have $S=R[\eta]_{\m\cap R[\eta]}$.

\end{lemma}

\begin{proof}
The fact that prime elements of $R$ remain prime in $S$ was noted in the proof of Lemma~\ref{3}.
So suppose $g\notin R$.
As $g$ is prime in $R[\eta]$, no height one prime ideal of $R$ contains all of the coefficients of $g$.
By construction, $\eta$ is transcendental over $R/P\cap R$ for every $P\in G-H$ and so $g\notin P$.
As $H$ is finite, it follows that we can find $r\in R$ such that  $g/r\in S$,  $g/r\notin P$ for every $P\in G$ and every height one prime ideal of $T$ which contains $r$ is in $H$.
Now, to determine the prime factorization of $g/r$, we can safely invert all nonzero elements of $R$.
This localization of $S$ must be an overring of the principal ideal domain $K[\eta]$, where $K$ is the quotient field of $R$, and so a localization of $K[\eta]$.
As $g/r$ is a prime element of $K[\eta]$, it is either a prime element or a unit in $S$.
Now let $L=R[\eta,t^{-1}]\cap S$.
Let $c\in S$.
As $S$ is contained in the quotient field of $R[\eta]$, we can write $c=a/b$ with $a,b\in R[\eta]$
and express both as a product of prime elements of $R[\eta]$.
We then alter the factorizations slightly, replacing each $g\in \m\cap R[\eta]$ by the corresponding $g/r$, which is an element of $L$.
 (We actually write $g=(g/r)r$, so $a,b$ do not change.)
After canceling common factors, $b$ is necessarily a unit in $S$ and since $a,b\in L$, $S$ is a localization of $L$.
The final statement is now obvious.
\end{proof}

\begin{defn}  If  $(R_1,\Gamma_1)$ and $(R_2,\Gamma_2)$ are N-pairs such that
 $R_1\subseteq R_2$, $\Gamma_1\subseteq \Gamma_2$, and $|R_1|=|R_2|$, we say $(R_2,\Gamma_2)$ is an A-extension of $(R_1,\Gamma_1)$.
\end{defn}

\begin{lemma}
If $(R,\Gamma)$ is an N-pair, then any application of Lemma~\ref{3} with $\Gamma\cup\{((x,y)T,\emptyset)\}\subseteq \Lambda$ will produce an N-subring $S$ such that $(S,\Gamma)$ is an A-extension of $(R,\Gamma)$.
\end{lemma}

\begin{proof}
We must show $(S,\Gamma)$ is an N-pair; certainly $|S|=|R|$.
To do this, we need only check the six conditions and conditions (1)-(4) are obvious.
Consider any $\lambda\in\Gamma$.
By lemma~\ref{describe}, $S$ is a localization of $R[\eta,t^{-1}]\cap S$ and since $t\notin P_{\lambda}$,
$S_{P_{\lambda}\cap S}$ is a localization of $R[\eta,t^{-1}]$ and so a localization of
$R_{P_{\lambda}\cap R}[\eta]$.
Let $K_{\lambda}$ denote the quotient field of $R/P_{\lambda}\cap R$.
As $\eta$ is transcendental over $K_{\lambda}$, it follows that the quotient field of
$S/P_{\lambda}\cap S$ is just  $K_{\lambda}(\eta)$.
The fact that $\eta$ is transcendental over $K_{\lambda}[\overline{V}_{\lambda}]$ tells us that
$\overline{V}_{\lambda}$ is algebraically independent over $K_{\lambda}(\eta)$ and so condition (5) holds.
Finally, since $S_{P_{\lambda}\cap S}$ is the localization of $R_{P_{\lambda}\cap R}[\eta]$
at the extension of the prime ideal $(P_{\lambda}\cap R)R_{P_{\lambda}\cap R}$,
condition (6) is also satisfied.
\end{proof}

\begin{lemma}\label {two}
Let  $(R,\Gamma)$ be an N-pair and suppose $Q\in \Spec T$ is a height two prime ideal.
If $Q\neq P_{\lambda}$ for all $\lambda\in\Gamma$, then
 there exists $S$, $R\subseteq S\subseteq T$, such that $(S,\Gamma)$ is an A-extension of $(R,\Gamma)$
  and $Q\cap R\nsubseteq P_{\lambda}$ for any $\lambda\in\Gamma$.
\end{lemma}

\begin{proof}
Let $\Lambda=\Gamma\cup\{((x,y)T,\emptyset)\}$, let $v=0$, and let $I=Q$.
Now apply Lemma~\ref{3} with $H=\emptyset$ to find an N-subring $S=R[d]_{\m\cap R[d]}$ which clearly satisfies the conclusion as $d\in Q-\bigcup P_{\lambda}$.
\end{proof}

\begin{lemma}\label {S}
Let  $(R,\Gamma)$ be an N-pair and $v\in T$.
Then there exists $S$, $R\subseteq S\subseteq T$, such that $(S,\Gamma)$ is an A-extension of $(R,\Gamma)$
and there exists $c\in S$ with $v-c\in \m^2$.
\end{lemma}

\begin{proof}
Let $\Lambda=\Gamma\cup\{((x,y)T,\emptyset)\}$ and let $I=\m^2$.
Now apply Lemma~\ref{3} with $H=\emptyset$ to find an N-subring $S$ which contains $c=v+d$ and the result is immediate.
\end{proof}

\begin{lemma}\label {4}
Let  $(R,\Gamma)$ be an N-pair, $c\in R$, $I$ a finitely generated ideal of $R$ and $c\in IT$.
Further suppose that $I\nsubseteq P_{\lambda}$ for all $(P_{\lambda},V_{\lambda})\in\Gamma$.
Then there exists $S$, $R\subseteq S\subseteq T$, such that $(S,\Gamma)$ is an A-extension of $(R,\Gamma)$
and $c\in IS$.
\end{lemma}
\begin{proof}
Let $\Lambda=\Gamma\cup\{((x,y)T,\emptyset)\}$.
If $I$ is contained in a height one prime ideal of $T$, because $R$ is a UFD,
 $I=bJ$ with $J$ not contained in a height one prime ideal of $T$.
 It suffices to prove the lemma with $J$ in place of $I$ and $b^{-1}c$ in place of $c$, and so we may assume
 $I$ is not contained in a height one prime ideal of $T$.
 Choose any nonzero $y_1\in I$.
 As $y_1\notin (x,y)T$, $y_1\in P_{\lambda}$ if and only if $P_{\lambda}$ is minimal over $(x,y,y_1)T$.
 Hence $y_1\in P_{\lambda}$ for at most finitely many $\lambda$.
 By the usual prime avoidance lemma, we can find $y_2\in I$ such that $(y_1,y_2)R$
 is contained in no $P_{\lambda}$ and in no height one prime ideal of $T$.
 We then choose $y_3,\ldots,y_n\in I$ so that $I=(y_1,\ldots,y_n)R$.
As $c\in IT$, we may write $c=\sum y_it_i$ with each $t_i\in T$.
We now apply Lemma~\ref{3} with $v=t_3$, $H=\emptyset$, and $(y_1,y_2)R$ playing the role of $I$.
Since $d\in (y_1,y_2)T$,
this allows us to find a new expression for $c$ as a linear combination of $y_1,\ldots,y_n$ with
$t_3$ replaced by $t_3+d$ and with suitably altered values of $t_1$ and $t_2$.
Hence, by enlarging $R$, we reduce to the case $t_3\in R$.
We proceed similarly with $t_4,\ldots,t_n$.
So $c-t_3y_3-\cdots-t_ny_n\in (y_1,y_2)T$ and so it will suffice to prove the lemma in the case $I=(y_1,y_2)R$.

The proof in this case closely resembles the proof of Lemma~\ref{3}.
Let $\Delta=\{P_{\lambda} | \lambda\in\Lambda\}$, $G=\{P\in  \Ass (T/rT)| r\in R\}$, and $C=\Delta\cup G$.
For each $P\in C$, we define a ring $A(P)$.
If $P\notin \Delta$, $A(P)=R/P\cap R$.
If $P=P_{\lambda}\in\Delta$, $A(P)=(R/P_{\lambda}\cap R)[\overline{V}_{\lambda}]$.
In each case, $A(P)$ has cardinality at most $|R|$ and so the same is true for the algebraic closure of $A(P)$ in $T/P$.
If $y_1\notin P$, let $D_P$ be a complete set of coset representatives modulo $P$ of $\{t\in T | t_2+ty_1\text{ is algebraic over } A(P)\}$.
If $y_1\in P$, then $y_2\notin P$ and we let $D_P$ be a complete set of coset representatives modulo $P$ of $\{t\in T | t_1-ty_2\text{ is algebraic over } A(P)\}$.
Therefore $|D_P|\leq|R|$.
Now $|C|\leq |R|$, and so also if $D=\bigcup_{P\in C}D_P$, then $|D|\leq|R|$, and finally $|C \times D|\leq|R|<|\mathbf{R}|$.
Hence we may apply Lemma~\ref {avoidance} to choose $t\in T$ so that either
the image of $t_2+ty_1$ in $T/P$ or the image of $t_1-ty_2$ in $T/P$  is transcendental over $A(P)$ for every $P\in C$.
Let $S$ be the intersection of $T$ with the quotient field of $R[t_2+ty_1]$, which is also
the intersection of $T$ with the quotient field of $R[t_1-ty_2]$ as $c=t_1y_1+t_2y_2$.
Clearly $c=y_1(t_1-ty_2)+y_2(t_2+ty_1)\in IS$.
Verifying that $(S,\Gamma)$ is an N-pair follows the same steps performed in the proof of Lemma ~\ref{3} and of course $|S|=|R|$.

\end{proof}

\begin{lemma}\label{intersection}
Let $R$ be an N-subring of $T$ with quotient field $K$.
Suppose $P\in\Spec T$ and $Q=P\cap R$.
Then $R_Q=T_P\cap K$.
\end{lemma}

\begin{proof}
Certainly, as $R-Q\subseteq T-P$, we have $R_Q\subseteq T_P\cap K$.
If the reverse containment fails, there exist $a,b\in R$ with $a/b\in T_P-R_Q$.
As $R$ is a UFD, we may express $a$ and $b$ as products of prime elements
and, canceling if necessary, we may assume $a,b$ have no common factors.
Further, we may delete any prime elements not contained in $Q$ as they will not affect
membership in either $R_Q$ or $T_P$.
Finally, let $p$ be any factor of $b$ --- one such must exist.
As $p\in Q$, $p\in P$, and so there exists a height one prime ideal $P_1$ of $T$
which is contained in $P$ and contains $p$.
Now $a/b\in T_P$ implies $a/p\in T_P\subseteq T_{P_1}$ which implies $a\in P_1\cap R=pR$.
But $p,a$ are relatively prime elements; this contradiction completes the proof.
\end{proof}

\begin{lemma}\label{union}
Let $(R_0,\Gamma_0)$ be an N-pair.
Let $\Omega$ be a well-ordered set with least element $0$ and no maximal element.
Suppose that for every $\alpha\in \Omega$, $|\{\beta\in\Omega |\beta <\alpha\}|<|\mathbf{R}|$.
Let $\gamma(\alpha)=\sup\{\beta\in\Omega | \beta<\alpha\}$.
Suppose $\{(R_{\alpha},\Gamma_{\alpha}) | \alpha\in\Omega\}$ is an ascending collection of pairs such that if $\gamma(\alpha)=\alpha$, then $R_\alpha=\bigcup_{\beta<\alpha}R_\beta$ and $\Gamma_\alpha=\bigcup_{\beta<\alpha}\Gamma_\beta$, while if $\gamma(\alpha)<\alpha$,
$(R_{\alpha},\Gamma_{\alpha})$ is an A-extension of $(R_{\gamma(\alpha)},\Gamma_{\gamma(\alpha)})$.
Then $(S,\Gamma)=(\bigcup R_{\alpha},\bigcup \Gamma_{\alpha})$ satisfies all conditions to be an N-pair except the cardinality condition.
In fact, $|S|\leq \sup (|R_0|,|\Omega|)$ and so $(S,\Gamma)$ is an N-pair if $|\Omega|<|\mathbf{R}|$.
\end{lemma}

\begin{proof}
We replace $\Omega$ by $\Omega^{\prime}=\Omega\cup \{\zeta\}$ with $\zeta>\alpha$ for all $\alpha\in\Omega$.
Let $(R_{\zeta},\Gamma_{\zeta})=(S,\Gamma)$.
The lemma then follows if we can show that for each $\alpha\in\Omega^{\prime}$:
$(R_{\alpha},\Gamma_{\alpha})$ satisfies all conditions to be an N-pair with the cardinality condition replaced by $|R_{\alpha}|\leq \sup(|R_0|,|\{\beta\in\Omega |\beta <\alpha\}|)$.
We do this by transfinite induction, the case $\alpha=0$ being trivial.

First note that since our proof will show that we have N-pairs all along the way, A-extensions
are permissible.
Now assume the induction hypothesis holds for all $\beta<\alpha$.
If $\gamma(\alpha)\neq\alpha$, it holds immediately for $(R_{\alpha},\Gamma_{\alpha})$
by the definition of A-extension.
Therefore assume $\gamma(\alpha)=\alpha$.
$R_{\alpha}$ is an N-subring of $T$ by \cite[Lemma 6]{UFD}).
(Alternatively, the reader may check the straightforward details directly.)
Certainly $\Gamma_{\alpha}$ has the form $\{(P_{\lambda},V_{\lambda}) | \lambda\in \Gamma\}$
and so it only remains to check the six conditions.
Conditions (1)-(4) are trivial.
For condition (5), note that if the images of the elements of $V_{\lambda}$ are not algebraically independent over
$R_{\alpha}/P_{\lambda}\cap R_{\alpha}$, there is a relation which necessarily involves
only finitely many elements of $R_{\alpha}$.
Hence there exists $\beta<\alpha$ such that each of these elements is contained in $R_{\beta}$
and $(P_{\lambda},V_{\lambda})\in \Gamma_{\beta}$.
This would contradict the fact that the images of the elements of $V_{\lambda}$ are algebraically independent over
$R_{\beta}/P_{\lambda}\cap R_{\beta}$ and so condition (5) holds.
Likewise, for condition (6), if $a\in Q_{\lambda}(R_{\alpha})_{Q_{\lambda}}$ we can find $\beta<\alpha$ such that
$a$ is in the quotient field of $R_{\beta}$ and $(P_{\lambda},V_{\lambda})\in \Gamma_{\beta}$.
Let $Q_{\lambda\beta}=P_{\lambda}\cap R_{\beta}$.
By lemma~\ref{intersection}, $(R_{\alpha})_{Q_{\lambda}}$ and $(R_{\beta})_{Q_{\lambda\beta}}$ are the intersection of $T_{P_{\lambda}}$ with their respective quotient fields and so
$a\in Q_{\lambda\beta}(R_{\beta})_{Q_{\lambda\beta}}$.
Thus we have
$a\in (x+\alpha_{\lambda}t_{\lambda}, y+\beta_{\lambda}t_{\lambda},t_{\lambda})(R_{\beta})_{Q_{\lambda\beta}}\subset (x+\alpha_{\lambda}t_{\lambda}, y+\beta_{\lambda}t_{\lambda},t_{\lambda})(R_{\alpha})_{Q_{\lambda}}$.
Therefore condition (6) holds as well.
\end{proof}

\begin{lemma}\label {u}
Let $R$ be an N-subring of $T$.
If $\omega,h\in T$ are such that $\overline{\omega}+\overline{z}\overline{h}$ is transcendental over
$R$ as an element of $T/(x,y)T$,
then there exists a subset $\Psi\subset\mathbf{R}$ such that
$|\Psi|\leq|R|$ and whenever $u\in\mathbf{R}-\Psi$, $\overline{\omega}+\overline{z}\overline{h}$ is transcendental over
$R/(x,y,z+uw)T\cap R$ as an element of $T/(x,y,z+uw)T$.

\end{lemma}

\begin{proof}
We choose $\Psi$ to be the set of real numbers that make the element algebraic and all that remains is to bound the cardinality of $\Psi$.
Now $u\in\Psi$ if $r_n(\omega+zh)^n+\cdots+r_0\in (x,y,z+uw)T$ for some $r_n,\ldots,r_0\in R$,
not all of which are contained in $(x,y,z+uw)T$.
Since the number of finite ordered subsets of $R$ is equal to $|R|$, we can restrict our attention to a single ordered subset, that is, we can fix $r_n,\ldots,r_0$ and ask when
$\delta=r_n(\omega+zh)^n+\cdots+r_0\in (x,y,z+uw)T$.
As $\overline{\omega}+\overline{z}\overline{h}$ is transcendental over $R$ as an element of $T/(x,y)T$,
$\delta\notin(x,y)T$.
Therefore $(x,y,z+uw)T$ must be minimal over $(x,y,\delta)T$, something that can be true for at most finitely many  values of $u$.
This completes the proof.
\end{proof}

\begin{lemma}\label{F}
Let  $(R,\Gamma)$ be an N-pair.
If $\omega,h\in T$ are such that $\overline{\omega}+\overline{z}\overline{h}$ is transcendental over
$R$ as an element of $T/(x,y)T$, then there exists $u\in \mathbf{ R}$ and a ring $S=R[u]_{\m\cap R[u]}$, such that $(x,y,z+uw)T\cap R=(0)$, $(S,\Gamma)$ is an A-extension of $(R,\Gamma)$,
and $\overline{\omega}+\overline{z}\overline{h}$ is transcendental over $R$ as an element of $T/(x,y,z+uw)T$.
\end{lemma}

\begin{proof}
First we use Lemma~\ref{u} to find a set $\Psi_1$ so that $u\notin\Psi_1$ will give
$\overline{\omega}+\overline{z}\overline{h}$ is transcendental over $R/(x,y,z+uw)T\cap R$ as an element of $T/(x,y,z+uw)T$.
Let $\Psi_2=\{u\in\mathbf{R} | (x,y,z+uw)T\cap R\neq (0)\}$.
As $(x,y)T\cap R=(0)$, no element of $R$ can be contained in infinitely many prime ideals
of the form $(x,y,z+uw)T$ and so $|\Psi_2|\leq |R|$.
Let $\Lambda=\Gamma\cup\{((x,y)T,\emptyset)\}$.
Now we use Lemma~\ref{3} with $v=0$, $I=R$, and $H=\emptyset$ to find the desired $S$.
Since $t=1$, we can choose $u=d\in \mathbf{R}-(\Psi_1\cup\Psi_2)$.
By Lemma~\ref{describe}, $S$ has the described form.

\end{proof}

\begin{lemma}\label {elem}
Let $t=z+uw$ for some $u\in T$ and let $P=(x,y,t)T=(r,s,t)T$ for some $r,s\in T$.
Suppose $F$ is a subfield of the quotient field of $T$, $\xi\in T$, and $E$ is the intersection of $T_P$ with the quotient field of $F[r,s,t,\xi]$.
Further assume $s^2\in F[r,t,\xi]+sF[r,t,\xi]$, $r^2+s^2\in tT$,
and $E=L_{P\cap L}$ where $L=F[r,t,\xi,s,t^{-1}]\cap T[F]$.
If there exists $\sigma_1,\sigma_2,\sigma_3\in T$ with $\sigma_1,\sigma_2$ algebraically independent over $F$ as elements of $T/P$ such that $\xi=\sigma_1r+\sigma_2s+\sigma_3t$, then
$E= F[r,s,t,\xi]_{(r,s,t,\xi)}$.
\end{lemma}

\begin{proof}
We first show that $F[r,s,t,\xi]\cap tT[F]=tF[r,s,t,\xi]$.
As $F[r,s,t,\xi]\subseteq T[F]$, one containment is clear.
For the reverse, let $g\in F[r,s,t,\xi]\cap tT[F]$.
As $s^2\in F[r,t,\xi]+sF[r,t,\xi]$, we may assume $g\in F[r,t,\xi]+sF[r,t,\xi]$ and since the terms involving $t$
are certainly in $tF[r,s,t,\xi]$, we may assume $g\in F[r,\xi]+sF[r,\xi]$.
In this case, we will show that $g=0$.
If $g\neq 0$, we can write $g=g_n+g_{n+1}+\cdots+g_m$ as a sum of homogeneous polynomials in $r,s,\xi$ with $g_n$ the nonzero homogeneous summand of lowest degree ($n$).
We have $g_n=-(g_{n+1}+\cdots+g_m)+g\in P^{n+1}T[F]+tT[F]$ and we will see by induction that no nonzero homogeneous polynomial of degree $n$ is in $P^{n+1}T[F]+tT[F]$.
This is obvious for $n=0$ as in that case $g_n\in F$.
For $n>0$, we can write $g_n(r,s,\xi)=f_1r^n+f_2sr^{n-1}+\xi h(r,s,\xi)$ where $f_1,f_2\in F$ and $h$ is a homogeneous polynomial of degree $n-1$.
As $\xi-(\sigma_1r+\sigma_2s)\in tT$.
we can replace $\xi$ by $\sigma_1r+\sigma_2s$, giving
$g_n(r,s,\sigma_1r+\sigma_2s)=f_1r^n+f_2sr^{n-1}+(\sigma_1r+\sigma_2s) h(r,s,\sigma_1r+\sigma_2s)\in P^{n+1}T[F]+tT[F]$.
Using the relation $r^2+s^2\in tT$, $g_n\equiv e_1r^n+e_2r^{n-1}s$ modulo $tT$ and necessarily
$e_1,e_2\in PT[F]$.
However, $e_1,e_2$ are polynomials over $F$ in $\sigma_1,\sigma_2$ with respective constant terms $f_1,f_2$.
As $\sigma_1,\sigma_2$ are algebraically independent over $F$, $e_1,e_2$ must both be identically zero and so $f_1=f_2=0$.
Therefore $g_n=h\xi$.
Next we take advantage of the fact that, modulo $tT$,
$\xi(\sigma_1r-\sigma_2s)\equiv \sigma_1^2r^2-\sigma_2^2s^2\equiv (\sigma_1^2+\sigma_2^2)r^2$ and $\sigma_1^2+\sigma_2^2\notin P$.
Now $g_n(\sigma_1r-\sigma_2s)\in (P^{n+1}+tT)(r,s)T[F]\subset(r^{n+2},r^{n+1}s,t)T[F]$.
So $h\xi(\sigma_1r-\sigma_2s)\equiv h(\sigma_1^2+\sigma_2^2)r^2\in (r^{n+2},r^{n+1}s,t)T[F]$.
Finally, as $(r^{n+2},r^{n+1}s,t)T$ is a primary ideal,
we have $h\in ((r^{n+2},r^{n+1}s,t)T:_{T[F]}r^2)=(r^n,r^{n-1}s,t)T[F]$ and by the induction assumption, $h=0$.
This contradiction shows that there can be no nonzero $g_n$ and so $F[r,s,t,\xi]\cap tT[F]= tF[r,s,t,\xi]$.

It follows that $L=F[r,s,t,\xi]$.
The result is now clear.
\end{proof}

\begin{lemma}\label{I}
Let $t=z+uw$ for some $u\in T$ and let $P=(x,y,t)T$
Let $F$ be a subfield of the quotient field of $T$ and suppose for $r,s,\xi_1\in T$ that $D=F[r,s,t,\xi_1]_{(r,s,t,\xi_1)}$ is the intersection of $T_P$ with the quotient field of $F[r,s,t,\xi_1]$.
Assume $P=(r,s,t)T$ and $r^2+s^2\in tT$.
Further assume
\begin{enumerate}
   \item  $s^2\in F[r,t]+\xi_1 F[r,t]+sF[r,t]$.
     \item $\delta_1,\delta_2,\tau,\sigma_1,\sigma_2\in T$ are algebraically independent over $F$ as elements of $T/P$.
  \item $\xi_1=\delta_1r+\delta_2s+\delta_3t$ with $\delta_3\in T$.
  \item $\delta_3=f(\delta_1,\delta_2)$ where $f(X_1,X_2)\in (F[t]\cap T)[X_1,X_2]$ is a polynomial in two indeterminates.
\end{enumerate}
Let $d_1=\delta_1+\sigma_1t$, $d_2=\delta_2+\sigma_2t$, and $d_3=\delta_3-\sigma_1r-\sigma_2s$.
Let $\xi_2=f(d_1,d_2)-d_3$.
Then if $E$ denotes  the intersection of $T_P$ with the quotient field of $D[d_1,d_2]$
and $E=L_{P\cap L}$ where $L=F[r,t,\xi_1,s,d_1,d_2,t^{-1}]\cap T[F[d_1,d_2]]$,
we have $E=F(d_1,d_2)[r,t,\xi_2,s]_{(r,t,\xi_2,s)}$.
Moreover we have
\begin{enumerate}
  \item $s^2\in F(d_1,d_2)[r,t]+\xi_2 F(d_1,d_2)[r,t]+sF(d_1,d_2)[r,t]$.
    \item $\sigma_1,\sigma_2,\tau\in T$ are algebraically independent over $F(d_1,d_2)$ as elements of $T/P$.
  \item $\xi_2=\sigma_1r+\sigma_2s+\sigma_3t$ with $\sigma_3\in T$.
  \item $\sigma_3=g(\sigma_1,\sigma_2)$ where $g(X_1,X_2)\in (F(d_1,d_2)[t]\cap T)[X_1,X_2]$ is a polynomial in two indeterminates.
\end{enumerate}

\end{lemma}

\begin{proof}
As $\delta_1,\delta_2$ are algebraically independent over $F$ as elements of $T/P$
and we have $d_i\equiv \delta_i$ modulo $P$ for each $i$, it follows that
$d_1,d_2$ are algebraically independent over $F$ as elements of $T/P$ and so
$F(d_1,d_2)\subset E$.
As $\xi_1=d_1r+d_2s+d_3t$, $d_3$ is in the quotient field of $E$, and so $d_3\in E$.
It follows that $\xi_2\in E$.
By taking the Taylor expansion of $f(d_1,d_2)$ about $(\delta_1,\delta_2)$,
we see that $\xi_2=-d_3+f(d_1,d_2)=-\delta_3+\sigma_1r+\sigma_2s+(f(\delta_1,\delta_2)+t\sigma_3)$
where $\sigma_3\in (F[t]\cap T)[\sigma_1,\sigma_2,\delta_1,\delta_2]=\\(F[t]\cap T)[\sigma_1,\sigma_2,d_1,d_2]\subset
(F(d_1,d_2)[t]\cap T)[\sigma_1,\sigma_2]$.
Conclusions (2),(3),(4) are now clear.
Also, as $\xi_1=d_1r+d_2s+d_3t\in F[d_1,r]+sF[d_2]+\xi_2F[t]+F[d_1,d_2,t]$, we get
$s^2\in F[r,t]+\xi_1 F[r,t]+sF[r,t]\subset F(d_1,d_2)[r,t]+\xi_2 F(d_1,d_2)[r,t]+sF(d_1,d_2)[r,t]$.
So conclusion (1) also holds.
Finally, as $d_3=f(d_1,d_2)-\xi_2$ and $\xi_1=d_1r+d_2s+d_3t$, we observe that $E$,
being a localization of $L$, is also a localization of $F(d_1,d_2)[r,s,t,\xi_2,t^{-1}]\cap T[F(d_1,d_2)]$.
We now apply Lemma~\ref{elem} to complete the proof.
\end{proof}

\begin{lemma}\label {P}
Let  $(R,\Gamma)$ be an N-pair.
Suppose $((x,y,z)T,\{\omega\})\in\Gamma$
and also suppose $h\in T$ such that $\omega+\overline{z}\overline{h}$ is transcendental over
$R$ as an element of $T/(x,y)T$.
Then there exist elements $u,\alpha,\beta, \mu,\nu\in\mathbf{R}$ and $r,s,t,A,B,C,q\in T$, and
a ring $S$, $R\subseteq S\subseteq T$, such that the following nine conditions hold.

\begin{enumerate}
\item $t=z+uw$
\item $r=x+\alpha t$
\item $s=y+\beta t$
\item $A=t\mu+rq-2\alpha(1+tq)$
\item $ B=t\nu+sq-2\beta(1+tq)$
\item $C=-r\mu-s\nu+(\alpha^2+\beta^2)(1+tq)$
\item $R[u,r,s,t,A,B,C]\subset S$
\item $q=0$
\item $(S,\Gamma\cup ((x,y,t)T,\{\omega +zh\})$ is an A-extension of $(R,\Gamma)$
\end{enumerate}
Suppose instead that $R=\mathbf{Q}[z,w]_{(z,w)}$ and $\Gamma=\emptyset$. Here we may set $u=0$ and find such a ring $S$ and an element $\omega\in \mathbf{R}$ satisfying conditions (1)-(7) with the last two conditions replaced by

\begin{enumerate}[start=8]

\item $q=\omega$
  \item  $(S,((x,y,z)T,\{\omega\}))$ is an A-extension of $(R,\Gamma)$
\end{enumerate}
\end{lemma}

\begin{proof}
To prove this result, we must choose elements $\alpha,\beta, \mu,\nu\in \mathbf{R}$, as well as either $u$ or $\omega$,
depending on the form of $R$.
Then we define $r,s,t,A,B,C,q$ so that Assertions (1-6) and Assertion (8) are satisfied.
Satisfying Assertion (7) is easy and so the only real obstacle is Assertion (9).
The basic plan is to obtain $S=\bigcup R_i$ with $(R_{i+1},\Gamma)$ an A-extension of $(R_i,\Gamma)$ for every $i$.
Then, by Lemma~\ref{union}, $(S,\Gamma)$ is an A-extension of $(R,\Gamma)$ and it only remains to adjoin one more element to $\Gamma$.
Throughout the proof, $P$ will denote the prime ideal $(x,y,t)T$ and $Q_i$ will denote $P\cap R_i$.
To verify that $(P,\{\omega+zh\})$ satisfies the necessary properties to allow it to be adjoined to $\Gamma$, we really do not need to understand $S$, only the simpler ring $S_{P\cap S}$.

We will deal with both cases simultaneously although we must note a few differences along the way.
In the $u=0$ case, as $|R/(x,y,z)T\cap R|<|T/(x,y,z)T|$, we may choose $\omega\in \mathbf{R}$ such that $\omega$ is transcendental over $R/(x,y,z)T\cap R$ as an element of  $T/(x,y,z)T$.
Then, by setting $h=0$ in the $u=0$ case, the two assertion (9)'s become identical.
For the transcendental case, we begin by applying Lemma~\ref{F} to get the appropriate element $u$ and an extension $R_1$ of $R$ satisfying all conclusions of that lemma;
in particular, $u\in R_1$ and $(R_1,\Gamma)$ is an N-pair.
For the $u=0$ case, we simply let $R_1=R$.
Either way, with $t=z+uw$ and $P=(x,y,t)T$, we see that $Q_1=P\cap R_1=tR_1$.
Moreover, if $F$ is the quotient field of $R$ in the transcendental case or $\mathbf{Q}(w)$ in the $u=0$ case, $(R_1)_{Q_1}=F[t]_{(t)}$.
We clearly have
$\omega+\overline{z}\overline{h}$ transcendental over $R_1/Q_1$ as an element of $T/P$.
Next we claim that $tT\nsubseteq P_{\lambda}$ for all $\lambda\in \Gamma$.
Certainly the only possible $P_{\lambda}$ which can contain it is $(x,y,t)T$.
In both cases, there is no such $P_{\lambda}$ by hypothesis.
Hereafter, except for the construction of $R_4$, the two cases will be handled identically.

Now set  $\Lambda=\Gamma\cup\{((x,y)T,\emptyset)\}$ and
$\Delta=\{P_{\lambda} | \lambda\in\Lambda\}$.
As $tT$ is a principal prime ideal of $T$, we can set $H=\{tT\}$ and then $tT\nsubseteq P^{\prime}$ for every $P^{\prime}\in \Delta\cup (G-H)$.
In  fact, we will use $H=\{tT\}$ and this choice of $\Lambda$ in every application of Lemma~\ref{3} contained in this proof.
Here we apply Lemma~\ref{3} with $v=x$ and $I=tT$ to find $d=\alpha t$ so that if $R_2$ is the intersection of the quotient field of $R_1[x+\alpha t]$ with $T$, $(R_2,\Gamma)$ is an N-pair.
Further, we may choose $\alpha\in\mathbf{R}$ so that $\alpha$ is transcendental over
$(R_1/Q_1)[\omega+\overline{z}\overline{h}]$ as an element of $T/P$.
Similarly we apply Lemma~\ref{3} to $R_2$ with $v=y$ and $I=tT$ to find $d=\beta t$ so that if $R_3$ is the intersection of 
the quotient field of $R_2[y+\beta t]$ with $T$, then $(R_3,\Gamma)$ is an N-pair.
We choose $\beta\in\mathbf{R}$ so that $\beta$ is transcendental over
$(R_2/Q_2)[\omega+\overline{z}\overline{h},\alpha]$ as an element of $T/P$ and
transcendental over $(R_2/tR_2)[\alpha]$ as an element of $T/tT$.
Therefore we have $u,r,s,t\in R_3$.

We have not yet chosen $\mu,\nu\in \mathbf{R}$, but it is useful at this time to perform a calculation which will be valid for any choice of $\mu,\nu$.
We claim that $r^2+s^2+Art+Brt+Ct^2=0$.
The verification of the claim is a straightforward calculation.
For future reference, we will actually perform the calculation in $\mathbf{ R}[[x,y,z,w]]$.
This makes sense as all of the elements involved are contained in
$\mathbf{R}+x\mathbf{R}+y\mathbf{R}+z\mathbf{R}+w\mathbf{R}\subset T$ and so have
natural liftings to $\mathbf{ R}[[x,y,z,w]]$.
After noting $r^2=x^2+2\alpha rt-\alpha^2t^2$ and $s^2=y^2+2\beta st-\beta^2t^2$, we have
\begin{align*}
&r^2+s^2+Art+Bst+Ct^2\\
&=(x^2+y^2)+rt(A+2\alpha)+st(B+2\beta)+t^2(C-(\alpha^2+\beta^2))\\
&=(x^2+y^2)+rt(t\mu+rq-2tq\alpha)+st(t\nu+sq-2tq\beta)
+t^2(-r\mu-s\nu+tq(\alpha^2+\beta^2))  \\
&=(x^2+y^2)+rt(rq-2tq\alpha)+st(sq-2tq\beta)+t^2(tq(\alpha^2+\beta^2))\\
&=(x^2+y^2)+tq(r^2-2rt\alpha+t^2\alpha^2+s^2-2st\beta+t^2\beta^2)\\
&=(x^2+y^2)(1+tq)
\end{align*}
This element is of course zero in $T$ since $x^2+y^2=0$.
For now, what this means is that $r^2+s^2\in tT$ and so $Y_3=-(r^2+s^2)/t\in R_3$.
We will next establish the properties of $R_3$ that we need.

As observed above, $Q_1=tR_1$ and
$\omega+\overline{z}\overline{h}$ is transcendental over $R_1/Q_1$ as an element of $T/P$.
Next we claim that $\overline{r}$ is transcendental over $R_1/tR_1$ as an element of $T/tT$.
If not, $\overline{x}$ is algebraic over $R_1/tR_1$ as an element of $T/tT$.
So we have $r_nx^n+\cdots+r_0\in tT$ with $r_n,\dots,r_0\in R_1$, $r_n\notin tR_1$, and $n$ minimal.
As $x\in P$, $r_0\in Q_1=tR_1$.
Thus the relation continues to hold if we replace $r_0$ by $0$.
However, now we can divide by $x$ and reduce the degree of the polynomial,
contradicting the minimality of $n$ and so proving the claim.
It follows from this and Lemma~\ref{describe} that $R_2=R_1[r]_{\m\cap R_1[r]}$
and so we have $(R_2)_{Q_2}=F[r,t]_{(r,t)}$.

We claim that $R_3$ is just a localization of $R_2[Y_3,s]$.
By Lemma~\ref{describe}, $R_3$ is a localization of a subring of $R_2[s,t^{-1}]$ and so of course
also a localization of a subring of $R_2[Y_3,s,t^{-1}]$.
Thus, we need only show $R_2[Y_3,s]\cap tT=tR_2[Y_3,s]$.
Using (4),(5),(6), which are valid equations in $T$ for any choice of $\mu,\nu$, we see that $Y_3=Ar+Bs+Ct=(r^2+s^2)q+(1+tq)(-2\alpha r-2\beta s+(\alpha^2+\beta^2)t)$  and so, as $r^2+s^2=-tY_3$, we have $Y_3=(1+tq)^{-1}(1+tq)(-2\alpha r-2\beta s+(\alpha^2+\beta^2)t)=-2\alpha r-2\beta s+(\alpha^2+\beta^2)t$.
We next show that $\overline{Y_3}$ is transcendental over $R_2/tR_2$ as an element of $T/tT$.
It suffices to show this with $\alpha r+\beta s$ in place of $Y_3$.
Since $r\in R_2$ and $s^2\equiv -r^2$ modulo $tT$, $\overline{s}$ is a nonzero element algebraic over $R_2/tR_2$.
Hence, if $\overline{Y_3}$ is algebraic over $R_2/tR_2$ and so also algebraic over $R_2/tR_2[\alpha]$, $\beta$ is algebraic over $R_2/tR_2[\alpha]$, contradicting our choice of $\beta$.
Of course, $\overline{Y_3}$ is also transcendental over $R_2/tR_2[\overline{s}]$ as an element of $T/tT$.
Now suppose $g\in R_2[Y_3,s]\cap tT$.
As $s^2\in R_2[Y_3]$, we may write $g=(a_ks+b_k)Y_3^k+\cdots+(a_0s+b_0)$ for some integer $k$
and each $a_i,b_i\in R_2$.
As $\overline{Y_3}$ is transcendental over $R_2/tR_2[\overline{s}]$, each $a_i+b_is\in tT$.
Now we will have $R_2[Y_3,s]\cap tT=tR_2[Y_3,s]$ if we can show $a_i,b_i\in tR_2=tT\cap R_2$.
If not, either $a_i$ or $b_i$ or both is not in $(r^m,t)T$ for sufficiently large $m$
and so is also not in $(r^m,t)R_2$ for sufficiently large $m$.
Choosing $m$ maximal so that $a_i,b_i\in (r^m,t)R_2$, we may write
$a_i=c_ir^m+\widetilde{c}_it$ and $b_i=d_ir^m+\widetilde{d}_it$.
As $(tT:_Tr^m)=tT$, $c_i+d_is\in tT$ and $c_i,d_i$ are not both contained in $(r,t)R_2=Q_2$.
But $c_i\in (s,t)T\cap R_2\subset Q_2=(r,t)R_2$ and this forces $d_i\in ((r,t):_{R_2}s)\subseteq Q_2$.
This contradiction proves each $a_i,b_i\in tR_2$ and
completes the proof of the claim that $R_3$ is just a localization of $R_2[Y_3,s]$.
This means then that $R_3$ is just a localization of $R_1[r,s,Y_3]$
and so $(R_3)_{Q_3}=F[r,s,t,Y_3]_{(r,s,t,Y_3)}$.

Next, let $A_3=rq-2\alpha(1+tq)$, $B_3=sq-2\beta(1+tq)$, and $C_3=(\alpha^2+\beta^2)(1+tq)$.
Since we chose $\alpha,\beta$ so that $\alpha,\beta,\omega+\overline{z}\overline{h}$  are algebraically independent over $F$ as elements of $T/P$, clearly
$\overline{A}_3,\overline{B}_3,\omega+\overline{z}\overline{h}$  are algebraically independent over $F$ as elements of $T/P$.
In the $q=0$ case, we also have $C_3=(1/4)(A_3^2+B_3^2)$.
Thus, for $n=3$ and $q=0$, with $F_3=F$, we have the following:
\begin{enumerate}
  \item $(R_n,\Gamma)$ is an N-pair
  \item $(R_n)_{Q_n}=F_n[r,s,t,Y_n]_{(r,s,t,Y_n)}$
    \item $Y_n=A_nr+B_ns+C_nt$ with $A_n,B_n,C_n\in T$ such that  $\overline{A}_n,\overline{B}_n,\omega+\overline{z}\overline{h}$  are algebraically independent over $F_n$ as elements of $T/P$
  \item $s^2\in F_n[r,t]+Y_nF_n[r,t]+sF_n[r,t]$

  \item $C_n\in (F_n[t]\cap T)[A_n,B_n]$
\end{enumerate}

In the $q=\omega$ case, (1)-(4) are satisfied.

We will construct an ascending union of N-subrings $R_3\subset R_4\subset R_5\subset \cdots$,
each of the same cardinality,
such that $R_n$ satisfies conditions (1)-(5) for every $n\geq 4$ along with the additional condition that
$Y_{n-1}\in (r,s,t)R_n$.
Note that Lemma~\ref{intersection} tells us that $(R_n)_{Q_n}$ is the intersection of
$T_P$ with the quotient field of $R_n$, a fact which allows us to freely use
Lemmas~\ref{elem} and \ref{I} in our construction process.

We now describe the construction of $R_n$.
We wish to adjoin elements $\widetilde{A}_n=A_{n-1}+t\sigma_{1n}$ and $\widetilde{B}_n=B_{n-1}+t\sigma_{2n}$.
To do this, we may apply Lemma~\ref{3} twice with $I=tT$ and $v=A_{n-1},B_{n-1}$ respectively to get
an extension $R_n$ with $(R_n,\Gamma)$ an A-extension of $(R_{n-1},\Gamma)$ and $\widetilde{A}_n,\widetilde{B}_n\in R_n$.
In defining these extensions, we can and do impose additional conditions on $\sigma_{1n}$ and $\sigma_{2n}$.
We choose $\sigma_{1n}\in \mathbf{R}$ to be transcendental over
$(R_{n-1}/Q_{n-1})[\omega+\overline{z}\overline{h},\overline{A}_{n-1},\overline{B}_{n-1}]$ as an element of $T/P$.
Likewise we choose $\sigma_{2n}\in \mathbf{R}$ to be transcendental over $(R_{n-1}/Q_{n-1})[\omega+\overline{z}\overline{h},\overline{A}_{n-1},\overline{B}_{n-1}, \sigma_{1n}]$ as an element of $T/P$.
Thus  $\omega+\overline{z}\overline{h},\overline{A}_{n-1},\overline{B}_{n-1},\sigma_{1n}, \sigma_{2n}$ is a set of algebraically independent elements over $R_{n-1}/Q_{n-1}$.
Since $\overline{A}_{n-1},\overline{B}_{n-1}$ are algebraically independent over $F_{n-1}$ as elements of $T/P$,
$F_n=F_{n-1}(\widetilde{A}_n,\widetilde{B}_n)$ is a field contained in $(R_n)_{Q_n}$.
We also note that if $\widetilde{C}_n=C_{n-1}-r\sigma_{1n}-s\sigma_{2n}$, $Y_{n-1}=\widetilde{A}_nr+\widetilde{B}_ns+\widetilde{C}_nt$ and so
$\widetilde{C}_n$ is in the intersection of $T$ with the quotient field of $R_n$, giving $\widetilde{C}_n\in R_n$
and thus $Y_{n-1}\in (r,s,t)R_n$.
By Lemma~\ref{describe}, $R_n$ is a localization of $R_{n-1}[\widetilde{A}_n,\widetilde{B}_n,t^{-1}]\cap T$.
So $(R_n)_{Q_n}$ is a localization of $F_{n-1}[r,s,t,Y_{n-1},\widetilde{A}_n,\widetilde{B}_n,t^{-1}]\cap T$.
If $R_{n-1}$ satisfies conditions (1)-(5), we get the full hypothesis of Lemma~\ref{I} when we consider the situation $D=(R_{n-1})_{Q_{n-1}}$, $\tau=\omega+zh$, $\xi_1=Y_{n-1}$, and
$\delta_1,\delta_2,\delta_3$ equal to $A_{n-1},B_{n-1},C_{n-1}$ respectively.
It immediately follows that $R_n$ also satisfies conditions (1)-(5).
Of course, $Y_n$ is the element $\xi_2$ specified in the conclusion of that lemma
and $A_n,B_n,C_n$ equal $\sigma_1,\sigma_2,\sigma_3$, respectively.

In the case $n=4,q=\omega$, we must check conditions (2)-(5) directly.
We define $Y_4=\widetilde{C}_4-(1/4)(\widetilde{A}_4^2+\widetilde{B}_4^2)$,
$A_4=q\alpha-\sigma_{14}$,  $B_4=q\beta-\sigma_{24}$, and
$C_4=(t/4)(A_4^2+B_4^2)+(1/2)(\widetilde{A}_4A_4+\widetilde{B}_4B_4)$.
Since $\widetilde{A}_4,\widetilde{B}_4\in F_4$, condition (5) is immediate.

To see condition (3), we simply calculate, first noting that
$\widetilde{A}_4+2A_4t=rq-2\alpha-t\sigma_{14}$ and $\widetilde{B}_4+2B_4t=sq-2\beta-t\sigma_{24}$.
\begin{align*}
&Y_4-A_4r-B_4s-C_4t\\
&=\widetilde{C}_4-(1/4)(\widetilde{A}_4^2+\widetilde{B}_4^2)-A_4r-B_4s-(t^2/4)(A_4^2+B_4^2)-(1/2)(\widetilde{A}_4A_4+\widetilde{B}_4B_4)t\\
&=\widetilde{C}_4-(1/4)(\widetilde{A}_4(\widetilde{A}_4+2A_4t)+\widetilde{B}_4(\widetilde{B}_4+2B_4t))-A_4r-B_4s-(t^2/4)(A_4^2+B_4^2)\\
&=(\alpha^2+\beta^2)(1+tq)-r\sigma_{14}-s\sigma_{24}-(1/4)(rq-2\alpha(1+tq)+t\sigma_{14})(rq-2\alpha-t\sigma_{14})\\
&-(1/4)(sq-2\beta(1+tq)+t\sigma_{24})(sq-2\beta-t\sigma_{24})-r(q\alpha-\sigma_{14})-s(q\beta-\sigma_{24})\\
&-(t^2/4)(\sigma_{14}^2+\sigma_{24}^2-2q(\sigma_{14}\alpha+\sigma_{24}\beta)+q^2(\alpha^2+\beta^2))\\
&=-(1/4)(r^2q^2-4\alpha rq-t^2\sigma_{14}^2-2\alpha tq(rq-t\sigma_{14})\\
&+s^2q^2-4\beta sq-t^2\sigma_{24}^2-2\beta tq(sq-t\sigma_{24}))\\
&-rq\alpha-sq\beta-(t^2/4)(\sigma_{14}^2+\sigma_{24}^2-2q(\sigma_{14}\alpha+\sigma_{24}\beta)+q^2(\alpha^2+\beta^2))\\
&=-(1/4)(r^2q^2-2\alpha tq(rq)+s^2q^2-2\beta tq(sq))-(t^2/4)(q^2(\alpha^2+\beta^2))\\
&=-(q^2/4)((r^2+s^2)-2t(\alpha r+\beta s)+t^2(\alpha^2+\beta^2))\\
&=-(q^2/4)(r^2+s^2+tY_3)=0
\end{align*}
The rest of condition (3) follows immediately from the construction.

For condition (4), we first note that $\widetilde{C}_4=Y_4+(1/4)(\widetilde{A}_4^2+\widetilde{B}_4^2)\in Y_4+F_4$.
Thus $Y_3=\widetilde{A}_4r+\widetilde{B}_4s+\widetilde{C}_4t\in rF_4+sF_4+tF_4+tY_4$.
Hence $s^2=-r^2-tY_3\in F_4[r,t]+Y_4F_4[r,t]+sF_4[r,t]$.
We also note $Y_3\in F_4[r,t,Y_4,s]$.
Finally, by Lemma~\ref{describe}, $R_4$ is a localization of $R_3[\widetilde{A}_4,\widetilde{B}_4, t^{-1}]\cap T$,
so $(R_4)_{Q_4}$ is a localization of $F_3[r,s,t,Y_3,\widetilde{A}_4,\widetilde{B}_4, t^{-1}]\cap T$,
which forces it to also be a localization of the larger ring $F_4[r,s,t,Y_4, t^{-1}]\cap T[F_4]$.
Now the entire hypothesis of Lemma~\ref{elem} is satisfied and we get condition (2).

Let $S=\bigcup R_n$; by Lemma~\ref{union}, $(S,\Gamma)$ is an A-extension of $(R,\Gamma)$.
Note that, referring to the construction of $R_4$, if we let $\mu=\sigma_{14}$, $\nu=\sigma_{24}$,
$A=\widetilde{A}_4$, $B=\widetilde{B}_4$, and $C=\widetilde{C}_4$, we have the desired $A,B,C\in S$.
Finally, to see that $(S,\Gamma\cup \{((x,y,t)T,\{\omega +zh\})\})$ is an N-pair, we let $Q=P\cap S$ and
we need only show $QS_Q=(r,s,t)S_Q$ and $\omega+\overline{z}\overline{h}$ is transcendental over $R/Q$ as an element of $T/P$.
As $QS_Q=(r,s,t,Y_3,Y_4,\cdots)S_Q$ and each $Y_n\in (r,s,t)S$, the first is true.
As $\omega+\overline{z}\overline{h}$ is transcendental over $R_n/Q_n$ as an element of $T/P$
for every $n$, the second is true as well.
\end{proof}

\bigskip
\bigskip

\begin{lemma}\label{primary}
Let $(S,\{((x,y,z)T,\{\omega\})\})$ be an N-pair constructed as in the proof of
Lemma~\ref{P} and let $m$ be any positive integer.
Then there exist $d_m,e_m\in S$ such that $(x,y,z^m)T=(d_m,e_m,z^m)T$.
\end{lemma}

\begin{proof}
For any $m>0$, we will find $d_m,e_m\in S$ such that
$x(1+z\omega/2)-d_m$ and $y(1+z\omega/2)-e_m$ are in $z^mT$.
That forces $d_m,e_m\in (x,y,z^m)T\cap S$ and clearly, as $1+z\omega/2$ is a unit,
$(d_m,e_m,z^m)T=(x,y,z^m)T$.
We will show that we can find $d_m$; the other case is identical.

To find $d_m$, we first observe
$x=r-z\alpha =r-(z/2)(-A+z\mu+r\omega-2\alpha z\omega)=
r-(z/2)(-A+z\mu+x\omega-\alpha z\omega)$ and so $x(1+z\omega/2)=r-(z/2)(-A+z(\mu-\alpha\omega))$.
Next recall from the construction that $\mu=\sigma_{14}=\alpha\omega-A_4$ and so
$x(1+z\omega/2)=r-(z/2)(-A+z(-A_4))=r+(z/2)(A+zA_4)$.
Now, for $n\ge 5$, $A_{n-1}+z\sigma_{1n}=\widetilde{A}_n\in S$ and $A_n=\sigma_{1n}$.
So $A_4=\widetilde{A}_5-zA_5=\widetilde{A}_5-z\widetilde{A}_6+z^2A_6$ and so on.
An easy induction shows $A_4\in S+z^mT$ for any $m$.
So $x(1+z\omega/2)=(r+zA/2)+(z^2/2)A_4\in S+z^mT$ for any $m$.
\end{proof}

\begin{lemma}\label{induction}
Let $(R,\Gamma)$ be an N-pair with $((x,y,z)T,\{\omega\})\in\Gamma$, and $h,v\in T$.
%DJ --- Added the word ``then" and a comma.
%RH   I added one more comma
If $\omega+\overline{z}\overline{h}$ is transcendental over $R$ as an element of $T/(x,y)T$, then there exists an
A-extension $(S,\Gamma^{\prime})$, where $\Gamma^{\prime}=\Gamma\cup\{(P,\{\omega+zh\})\}$
for some $P\in \Spec T$, and an element $c\in S$
such that
\begin{enumerate}
  \item $v-c\in \m^2$
  \item  For every finitely generated ideal $I$ of $R$ such that
$I\nsubseteq P_{\lambda}$ for all $\lambda\in\Gamma^{\prime}$, we have $IT\cap R\subset IS$.
  \item If $P^{\prime}$ is a height two prime ideal of $T$ which is minimal over $IT$
for some finitely generated ideal $I$ of $R$ and $P^{\prime}\neq P_{\lambda}$ for all $\lambda\in\Gamma^{\prime}$, then $P^{\prime}\cap S \nsubseteq P_{\lambda}$ for all $\lambda\in\Gamma^{\prime}$.
\end{enumerate}

If $\omega+\overline{z}\overline{h}$ is algebraic over $R$ as an element of $T/(x,y)T$
(a case we will later see cannot occur), there exists an A-extension $(S,\Gamma)$ and an element $c\in S$ such that conditions (1)-(3) above hold.

\end{lemma}
\begin{proof}
In the transcendental case, we first employ Lemma ~\ref{P} to obtain an A-extension $(R_0,\Gamma^{\prime})$ of $(R,\Gamma)$
where $\Gamma^{\prime}=\Gamma\cup\{(P,\{\omega+zh\})\}$ for some $P\in\Spec T$ of the proper form.
In the algebraic case, we simply let $(R_0,\Gamma^{\prime})=(R,\Gamma)$ and for the remainder of the proof, the two cases are identical.
Next employ Lemma ~\ref{S} to obtain an A-extension $(R_1,\Gamma^{\prime})$ of $(R_0,\Gamma^{\prime})$ with $c\in R_1$ such that $v-c\in \m^2$.
Set  $\Omega_1=\{(I,d)|  I \text{ finitely generated ideal of }R \text{
such that } I\nsubseteq P_{\lambda} \text{ for all } \lambda\in\Gamma^{\prime} \text{  and } \\d\in IT\cap R\}$.
Set $\Omega_2=\{P^{\prime}\in\Spec T | \Ht P^{\prime}=2, P^{\prime}\neq P_{\lambda} \text{ for all }\lambda\in\Gamma^{\prime},
 \text{ and }\\P^{\prime} \text{ minimal over } IT\text{ for some finitely generated ideal } I\text{ of }R\}$.
Set $\Omega=\Omega_1\cup\Omega_2$.

Clearly $|\Omega|=|R|$.
Well-order $\Omega$, letting $1$ denote its initial element, in such a way that $\Omega$ does not have a maximal element; then it satisfies the hypothesis of Lemma ~\ref{union}.
Next we recursively define a family $\{R_{\alpha} | \alpha\in\Omega \}$ to satisfy the hypothesis of Lemma ~\ref {union}.
We begin with $R_1$.
If $\gamma(\alpha)\neq\alpha$ and $\gamma(\alpha)=(I,d)$, then we choose $R_{\alpha}$ to be an A-extension of $R_{\gamma(\alpha)}$ given by Lemma ~\ref{4} such that $d\in IR_{\alpha}$.
If $\gamma(\alpha)\neq\alpha$ and $\gamma(\alpha)=P^{\prime}$, then we choose $R_{\alpha}$ to be an A-extension of $R_{\gamma(\alpha)}$ given by Lemma ~\ref{two} such that
$P^{\prime}\cap R\nsubseteq P_{\lambda}$ for all $\lambda\in\Gamma^{\prime}$.
If $\gamma(\alpha)=\alpha$, choose $R_{\alpha}=\bigcup_{\beta<\alpha}R_{\beta}$.
 Set $S=\bigcup R_{\alpha}$.
By Lemma ~\ref{union} and the observation that $|\Omega| = |R |$, we see that
$(S,\Gamma^{\prime})$ is an A-extension of $(R_1,\Gamma^{\prime})$.
Also, if $I$ is any finitely generated ideal of $R$
such that  $I\nsubseteq P_{\lambda}$ for all $\lambda\in\Gamma^{\prime}$ and $d\in IT\cap R$, then
$(I,d)=\gamma(\alpha)$ for some $\alpha\in \Omega$.
So $d\in IR_{\alpha}\subset IS$.
Thus $IT\cap R\subset IS$.
Likewise, if $P^{\prime}\neq P_{\lambda}$ is a height two prime ideal of $T$
which is minimal over a finitely generated ideal of $R$,
$P^{\prime}\cap S$ is not contained in any $P_{\lambda}$.

\end{proof}

\begin{lemma}\label{Noeth}
Suppose $(R,\Gamma)$ is an ascending union of N-pairs.
Assume that the natural map $R\to T/\m^2$ is surjective and that for every finitely generated ideal $I$ of $R$ such that $I\nsubseteq P_{\lambda}$ for all $\lambda\in\Gamma$, $IT\cap R=I$.
Further assume that if $P^{\prime}$ is a height two prime ideal of $T$ which is minimal over $IT$
for some finitely generated ideal $I$ of $R$ and $P^{\prime}\neq P_{\lambda}$ for all $\lambda\in\Gamma$, $P^{\prime}\cap R \nsubseteq P_{\lambda}$ for all $\lambda\in\Gamma$.
Then $R$ is Noetherian and the natural homomorphism $\widehat{R}\to T$ is an isomorphism.
\end{lemma}

\begin{proof}
By Proposition 1 of \cite{sing}, it suffices to prove $IT\cap R=I$ for every finitely generated ideal $I$ of $R$ in order to get our conclusion.
Suppose the statement is false; we may choose
a finitely generated ideal $J$ such that $JT\cap R\neq J$ and $JT$ is maximal among such ideals.
By our hypothesis, $JT\subseteq P_{\lambda}$ for some $\lambda$.
As no $(\m\cap R)$-primary ideal is contained in any $P_{\lambda}$,
we have $IT\cap R=I$ whenever $I$ is primary to the maximal ideal.
Thus we may apply Lemma 21 of \cite{sing} to see that $JT\cap R$ is infinitely generated.
Also, if I is an ideal of $R$ which is infinitely generated and properly contains $JT\cap R$,
there is necessarily a finitely generated ideal $I_0\subset I$ such that $I\subset I_0T$.
Then $I_0T\cap R\neq I_0$, contradicting the maximality of $JT$.
Hence $JT\cap R$ is maximal among infinitely generated ideals of $R$.
Cohen's theorem tells us that $JT\cap R$ is a prime ideal $Q$.
It is easy to see that it is the contraction of one or more of the associated prime ideals of $JT$.
As the contraction of height one prime ideals are principal, it is the contraction of a height two
prime ideal which is minimal over $JT$.
By the hypothesis, if $P^{\prime}\neq P_{\lambda}$ for all $\lambda$,
$P^{\prime}\cap R$ cannot be a height two ideal contained in some $P_{\lambda}$.
However, necessarily $J$, and so also $JT$, is contained in some $P_{\lambda}$.
So $Q=P_{\lambda_0}\cap R$ for some $\lambda_0\in\Gamma$.

For notational ease, let $P=P_{\lambda_0}$ and $Q=P\cap R$.
Using the notation from the statement of Lemma~\ref{P}, $Q$ is locally generated by $r,s,t$ and $QT$ is generated by $r,s,t$.
We will show that $Q=(r,s,t)R$ and so contradict the notion that $Q$ is infinitely generated,
thus completing the proof.
Suppose $h\in Q$; we must show $h\in (r,s,t)R$.
There exists $d\in R-Q, e_1,e_2,e_3\in R$ such that $dh=e_1r+e_2s+e_3t$.
We also have $\sigma_1,\sigma_2,\sigma_3\in T$ such that $h=\sigma_1r+\sigma_2s+\sigma_3t$.
Thus $(e_1-d\sigma_1)r+(e_2-d\sigma_2)s+(e_3-d\sigma_3)t=0$.
It follows that $e_2-d\sigma_2\in  ((r,t)T:_Ts)=P$ and so $e_2\in P+dT=(r,s,t,d)T$.
As $(r,s,t,d)R$ is $(\m\cap R$)-primary, $e_2\in (r,s,t,d)T\cap R= (r,s,t,d)R$.
We can write $e_2=\rho_1^{\prime}r+\rho_2s+\rho_3d+\rho_4^{\prime}t$ with $\rho_1^{\prime},\rho_2,\rho_3,\rho_4^{\prime}\in R$ and setting
$\rho_1=e_1+s\rho_1^{\prime},\rho_4=e_3+s\rho_4^{\prime}$, we have
$dh=\rho_1r+(\rho_2 s+\rho_3 d)s+\rho_4t$ with $\rho_1,\rho_2,\rho_3,\rho_4\in R$.
Now it suffices to show $h-\rho_3 s\in (r,s,t)R$ and so we can reduce to the case $\rho_3=0$.
Then, as $s^2\in (r,t)R$, we can write $dh=\delta_1 r+\delta_2 t$ with $\delta_1,\delta_2\in R$.
Next observe that $P$ is the unique height two prime ideal of $T$ containing $(r,t)T$
and as $d\notin P$, $r,t,d$ is a system of parameters.
Therefore, as $T$ is Cohen-Macaulay, $\delta_1\in (d,t)T\cap R$.
Now $d\notin P_{\lambda_0}$ and $t\notin P_{\lambda}$ for $\lambda\neq\lambda_0$,
giving $(d,t)T\cap R=(d,t)R$ and $\delta_1\in(d,t)R$.
Therefore $dh=(\delta_3 d+\delta_4 t)r+\delta_2 t$ with $\delta_3,\delta_4\in R$.
Finally $d(h-\delta_3 r)\in tR$ and so $h\in (r,t)R$ as desired.
\end{proof}

\begin{thm}
There exists a local UFD $R$ with completion $T$ such that
$(R,\Gamma)=\bigcup(R_{\alpha},\Gamma_{\alpha})$ is an ascending union of N-pairs
which satisfies every defining condition of N-pair except that $|R|=|\mathbf{R}|$.
Moreover, there is an element $\omega\in \mathbf{R}$ such that for every $h\in T$, there exists an element
$(P,\{\omega+zh\})\in\Gamma$ constructed using Lemma~\ref{P}.
In particular, we have $((x,y,z)T,\{\omega\})\in\Gamma$.

\end{thm}

\begin{proof}
We start with the N-subring $R_0=\mathbf{Q}[z,w]_{(z,w)}$ and note that $(R_0,\emptyset)$ is an N-pair.
Then we apply Lemma~\ref{P} to obtain an element $\omega\in \mathbf{R}$ and an N-subring $R_1$ such that
$(R_1,\{((x,y,z)T,\{\omega\})\})$ is an N-pair.
Next let $\Omega$ be the set $T\times T$, well-ordered so that $\forall\alpha\in \Omega$, $|\{\beta\in\Omega |\beta <\alpha\}|<|\mathbf{R}|$.
Let $\gamma(\alpha)=\sup\{\beta\in\Omega | \beta<\alpha\}$.
Define an ascending collection of pairs $\{(R_{\alpha},\Gamma_{\alpha}) | \alpha\in\Omega\}$ so that if $\gamma(\alpha)=\alpha$, then $R_\alpha=\bigcup_{\beta<\alpha}R_\beta$ and $\Gamma_\alpha=\bigcup_{\beta<\alpha}\Gamma_\beta$, while if $\gamma(\alpha)<\alpha$, then
$(R_{\alpha},\Gamma_{\alpha})$ is an A-extension of $(R_{\gamma(\alpha)},\Gamma_{\gamma(\alpha)})$ given by Lemma~\ref{induction} so that if $\gamma(\alpha)=(h,v)$,
then $v\in R_{\alpha}+\m^2$ and $\Gamma(\alpha)=\Gamma(\gamma(\alpha))\cup\{(P_{\alpha},\{\omega+zh\})\}$ whenever $\omega+\overline{z}\overline{h}$ is transcendental over $T/(x,y)T$.
Of course, if $\omega+\overline{z}\overline{h}$ is algebraic over $T/(x,y)T$, a case we will rule out later in the proof, then $v\in R_{\alpha}+\m^2$ and $\Gamma(\alpha)=\Gamma(\gamma(\alpha))$.

Let $R=\bigcup R_{\alpha}$ and $\Gamma=\bigcup \Gamma_{\alpha}$.
By Lemma~\ref{union}, $(R,\Gamma)$ satisfies all conditions to be an N-pair except the cardinality condition.
We still must show that $R$ is a Noetherian ring with completion $T$ to complete the proof of the first statement.
We do this by showing that $R$ satisfies the entire hypothesis of Lemma~\ref{Noeth}.
We begin by noting that it is clear from the construction that the natural
map $R\to T/\m^2$ is surjective.
Next we claim that if $I$ is any finitely generated ideal of $R$ such that $I\nsubseteq P_{\lambda}$ for all $\lambda\in\Gamma$, then $IT\cap R=I$.
To prove the claim, let $I=(y_1,\ldots,y_n)R$ and suppose $d\in IT\cap R$.
As $y_1,\ldots,y_n,d$ is a finite set, there exists $\alpha$ such that $y_1,\ldots,y_n,d\in R_{\alpha}$.
By Lemma~\ref{induction}, we get $d\in (y_1,\dots,y_n)R_{\alpha+1}$ and so $IT\cap R=I$,
proving the claim.
Finally let $P^{\prime}$ be a height two prime ideal of $T$ which is not equal to any $P_{\lambda}$
but is minimal over $IT$ for some finitely generated ideal $I$ of $R$, and
consider any specific $\lambda\in\Gamma$.
Again we let $I=(y_1,\ldots,y_n)R$ and note that there exists $\alpha$ such that $y_1,\ldots,y_n\in R_{\alpha}$ and $(P_{\lambda},V_{\lambda})\in \Gamma_{\alpha}$.
By Lemma~\ref{induction}, we get $P^{\prime}\cap R_{\alpha+1} \nsubseteq P_{\lambda}$.
Now Lemma~\ref{Noeth} tells us that $R$ is Noetherian and the natural homomorphism $\widehat{R}\to T$ is an isomorphism.

Finally we show that for each $h\in T$ there exists some $(P,\{\omega+zh\})\in\Gamma$.
If not, when we encountered $(h,v)$ in the recursive process, we discovered that
$\omega+\overline{z}\overline{h}\in T/(x,y)T$ was algebraic over the subring $R_{\alpha}$ in use at that time.
However, as $R_{\alpha}\subset R$, $\omega+\overline{z}\overline{h}\in T/(x,y)T$ is also algebraic over $R$.
We complete the proof by contradicting the assumption that
$\omega+\overline{z}\overline{h}\in T/(x,y)T$ is algebraic over $R$.

The algebraic assumption gives elements $r_n,\dots,r_0\in R$, not all zero, such that
$r_n(\omega+zh)^n+\cdots+r_0\in (x,y)T$.
As $(x,y)T\cap R=(0)$, $(r_n,\dots,r_0)R\nsubseteq (x,y,z^m)T$ for sufficiently large $m$.
Choose $m$ minimal with respect to the property that $(r_n,\dots,r_0)R\nsubseteq (x,y,z^m)T$.
By Lemma~\ref{primary}, there exists $d,e\in R$ such that $(d,e,z^m)T=(x,y,z^m)T$.
As ideals in $R$ are closed, $(x,y,z^m)T\cap R=(d,e,z^m)R$ and
$(x,y,z^{m-1})T\cap R=(d,e,z^{m-1})R$.
Let $J$ denote $(d,e,z^m)T$.
For each $j$, as $r_j\in (x,y,z^{m-1})T$, we can write $r_j=a_jd+b_je+c_jz^{m-1}$
with $a_j,b_j,c_j\in R$.
Since $r_n(\omega+zh)^n+\cdots+r_0\in J$ and $d,e\in J$,
$c_nz^{m-1}(\omega+zh)^n+\cdots+c_0z^{m-1}\in J$.
Thus $z^{m-1}(c_n\omega^n+\cdots+c_0)\in (x,y,z^m)T$
and so $c_n\omega^n+\cdots+c_0\in (x,y,z)T$.
However, by the choice of $m$, at least one $c_j\notin (x,y,z)T$.
This contradicts the fact that $\omega$ is transcendental over $R/Q$.
\end{proof}

\begin{thm}
Let $R$ be as constructed.  Then $R$ is not the homomorphic image of a regular local ring.
\end{thm}
\begin{proof}
Assume the theorem is false.
By the theorem in the introduction, we have a commutative diagram
\[
\xymatrixrowsep{2pc}
\xymatrixcolsep{5pc}
\xymatrix{
S \ar@{.>}[r]^\subseteq\ar@{.>>}[d]_{\pi\vert_S} & \mathbf{R}[[x,y,z,w]] \ar@{->>}[d]^{\pi}\\
R\ar[r]^\subseteq & \mathbf{R}[[x,y,z,w]]/(x^2+y^2) \\
}
\]

Let $f=x^2+y^2\in \mathbf{R}[[x,y,z,w]]$.
Throughout the construction of $R$, we chose all elements of interest to be members of the
vector space $\mathbf{R}+x\mathbf{R}+y\mathbf{R}+z\mathbf{R}+w\mathbf{R}\subset T$.
For these elements, there is a natural lifting to $\mathbf{R}[[x,y,z,w]] $.
Of course, these canonical lifts are unlikely to be in $S$, but we employ them to simplify notation.
So for example, if $A^L$ denotes a lifting of $A$, we can write $A^L=A+fA_0$ for some
$A_0\in \mathbf{R}[[x,y,z,w]] $.

Consider $((x,y,t_{\lambda})T,\{\omega+zh_{\lambda}\})\in \Gamma$.
We have $P_{\lambda}=(x,y,t_{\lambda})T=(r_{\lambda},s_{\lambda},t_{\lambda})T$ with
$t_{\lambda}=z+u_{\lambda}w$, $r=x+\alpha_{\lambda}t_{\lambda}$, and $s=y+\beta_{\lambda}t_{\lambda}$.
As nearly everything depends on $\lambda$, we will suppress the subscript in our calculations.
We begin by lifting the equation $0=r^2+s^2+Art+Bst+Ct^2$ to the full power series ring.
Let $A^L$ denote the lift of $A$, etc. and we get an element $\Theta$ in the kernel of $\pi|_S$
such that
$\Theta=(r^L)^2+(s^L)^2+A^Lr^Lt^L+B^Ls^Lt^L+C^L(t^L)^2=(r+f\rho)^2+(s+f\sigma)^2+
(A+fA_0)(r+f\rho)(t+f\tau)+(B+fB_0)(s+f\sigma)(t+f\tau)+(C+fC_0)(t+f\tau)^2=
r^2+s^2+Art+Bst+Ct^2+f(2\rho r+2\sigma s +Ar\tau+At\rho+A_0rt+Bs\tau+Bt\sigma+B_0st
+2Ct\tau+C_0t^2)+f^2d$ where $\rho,\sigma,\tau,A_0,B_0,C_0\in T$ and
$d\in T$ is a sum of terms which we don't need to keep track of individually.
Recall that in the proof of Lemma~\ref{P}, we actually did the calculation in $\mathbf{ R}[[x,y,z,w]]$
and determined that $r^2+s^2+Art+Bst+Ct^2=f(1+tq)$ and so
$\Theta=f(1+tq+(2r+At)\rho+(2s+Bt)\sigma+(Ar+Bs+2Ct)\tau+A_0rt+B_0st+C_0t^2+fd)$.
As $r=x+\alpha t$ and $s=y+\beta t$,
$\Theta\in f(1+tq+(2\alpha+A)t\rho+(2\beta+B)t\sigma+(A\alpha+B\beta+2C)t\tau+(x,y,t^2)T)$.
Also, from the defining equations for $A,B,C$, we see that
$2\alpha+A, 2\beta+B, A\alpha+B\beta+2C\in (x,y,t)T$ and so
$\Theta\in f(1+tq+(x,y,t^2)T)$.

Now we consider the specific $((x,y,z)T,\{\omega\})\in\Gamma$, which we will refer to as $\lambda=0$.
As $q_0=\omega$, we have $\Theta_0\in f(1+z\omega+z^2h+(x,y)T)$ for some $h\in T$.
There is another element in $\Gamma$ corresponding to $h$ which it is convenient to denote $\lambda=h$.
(Admittedly, we need a different name if $h$ happens to equal zero.)
This is $(P_h,\{\omega+zh\})$.
Clearly $\Theta_h\in f(1+P_h)$.
As both $\Theta_0$ and $\Theta_h$ are elements of $S$ and they are unit multiples of each other,
we see that $\Theta_0/\Theta_h\in S$.
The map $S\to R\to R/P_h\cap R$ clearly takes $\Theta_0/\Theta_h$ to $1+\overline{z}\omega+\overline{z}^2\overline{h}$.
Thus $\overline{z}(\omega+\overline{z}\overline{h})\in R/P_h\cap R$.
This contradicts the fact that $\omega+\overline{z}\overline{h}$ is transcendental over $R/P_h\cap R$,
completing the proof.

\end{proof}

\end{document}